\documentclass[12pt]{amsart}
\usepackage{times, pdfpages, graphicx, bbm}
\usepackage{comment}
\usepackage{exscale}
\usepackage{tikz-cd}
\usepackage{epsfig}

\usepackage{amsthm}
\usepackage{amsfonts}

\usepackage{graphics,psfrag,graphicx,subfigure,epsfig, xypic}
\usepackage{amsmath,amsfonts,amssymb,amstext,amsthm,amscd,mathrsfs}
\usepackage[english]{babel}
\usepackage[all]{xy}
\usepackage{multicol}
\usepackage{rotating}

\usepackage{amsmath, amsfonts, amssymb, amscd, enumerate}
\usepackage{color}
\usepackage[all]{xy}

\theoremstyle{plain}
\newtheorem{theo}{Theorem}[section]
\newtheorem{lem}[theo]{Lemma}

\theoremstyle{definition}

\newtheorem{example}[theo]{Example}
\newtheorem{definition}[theo]{Definition}

\theoremstyle{plain}

\newtheorem{theorem}[theo]{Theorem}

\newtheorem{proposition}[theo]{Proposition}

\theoremstyle{definition}

\newtheorem{remark}[theo]{Remark}

\newcommand{\beq}{\begin{equation}}
\newcommand{\eeq}{\end{equation}}

\newcommand{\ve}{\varepsilon}

\renewcommand{\i}{\iota}

\newcommand{\C}{\mathbb{C}}

\newcommand{\R}{\mathbb{R}}
\renewcommand{\H}{\mathbb{H}}
\newcommand{\bH}{\mathbb{H}}
\newcommand{\Z}{\mathbb{Z}}

\newcommand{\bN}{\mathbb{N}}

\newcommand{\bS}{\mathbb{S}}

\newcommand{\cI}{\mathcal{I}}

\renewcommand{\cR}{\mathcal{R}}
\newcommand{\cS}{\mathcal{S}}

\newcommand{\ra}{\rightarrow}

\def\Log{\mathop\text{\rm log}\nolimits}

\renewcommand{\square}{\kern1pt\vbox
{\hrule height 0.6pt\hbox{\vrule width 0.6pt\hskip 3pt
\vbox{\vskip 6pt}\hskip 3pt\vrule width 0.6pt}\hrule height0.6pt}\kern1pt}

\DeclareMathOperator\Id{Id}

\renewcommand\Im{\operatorname{Im}}

\renewcommand{\Im}{{\rm Im}}
\renewcommand{\min}{{\rm min}\,}
\newcommand\refl[1]{R(#1)}

\newcommand{\be}{\begin{equation}}
\newcommand{\ee}{\end{equation}}

\def\<#1,#2>{\langle\,#1,\,#2\,\rangle}
\newcommand{\arr}{\begin{array}{rlll}}
\newcommand{\ea}{\end{array}}
\newcommand{\bea}{\begin{eqnarray}}
\newcommand{\eea}{\end{eqnarray}}
\newcommand{\bean}{\begin{eqnarray*}}
\newcommand{\eean}{\end{eqnarray*}}

\newcommand{\ssd}{symmetric slice domain }

\def\sideremark#1{\ifvmode\leavevmode\fi\vadjust{
\vbox to0pt{\hbox to 0pt{\hskip\hsize\hskip1em
\vbox{\hsize3cm\tiny\raggedright\pretolerance10000
\noindent #1\hfill}\hss}\vbox to8pt{\vfil}\vss}}}

\newcounter{ssig}
\newcounter{ttig}
\begin{document}
\title[Vanishing of quaternionic cohomology groups and applications]{Vanishing of quaternionic cohomology groups and applications}
\
\author{Jasna Prezelj}
\address{UL FMF, Jadranska 19,
  Ljubljana, Slovenija, UP FAMNIT, Glagolja\v ska 8,  Koper, Slovenija, IMFM, Jadranska 19, Ljubljana, Slovenija}
\email { jasna.prezelj@fmf.uni-lj.si}
\author {Fabio  Vlacci}\address{MIGe Universit\`a di Trieste Piazzale Europa 1,\ Trieste,
  Italy} \email{ fvlacci@units.it} \thanks{\rm The first author
  was partially supported by research program P1-0291 and by research
  projects N1-0237, J1-3005 at Slovenian Research Agency.  The second author was partially supported by
  Progetto MIUR di Rilevante Interesse Nazionale PRIN 2010-11 {\it
    Variet\`a reali e complesse: geometria, topologia e analisi
    armonica}.  The research that led to the present paper was
  partially supported by a grant of the group GNSAGA of Istituto
  Nazionale di Alta Matematica `F: Severi'.}


\maketitle
\begin{abstract}
We present  solutions to additive and multiplicative Cousin problems formulated on an axially symmetric domain $\Omega \subset \H$ for  slice--regular functions  starting from the solutions for  subclasses, namely slice--regular slice--preserving functions and functions in a given vectorial class. As a consequence, we prove the vanishing of the corresponding cohomology groups with respect to axially symmetric open coverings (Theorems \ref{H1*}, \ref{H^1SRR+}, \ref{H^1SRR}).  The main tool used in the proofs of these theorems is the existence of quaternionic Cartan coverings and Cartan's splitting lemmas. 
As an application we prove a jet interpolation theorem (Theorem \ref{jets1}) and we show that every divisor is principal (Theorem \ref{divisor}).
\end{abstract}

\section{Introduction}\label{intro}
The theory of slice--regular functions (denoted by $\mathcal{SR}$ and shortly recalled in the
Section \ref{prelim} and mainly focussed on specific aspects useful later)
clearly shows many promising aspects to become a good framework to
consider the generalizations of Cousin problems for quaternionic
functions; indeed, the notion of slice--regularity is a good extension
of the notion of holomorphicity for quaternionic functions and
semi-regular functions play the role of meromorphic functions.  The
analogue of Weierstrass and Mittag-Leffler Theorems are already at our
disposal (see \cite{GSS}). Furthermore, domains of holomorphicity are
generalized in terms of quaternionic axially symmetric domains.

In the present paper we prove the analogues of the above theorems for arbitrary axially symmetric domains $\Omega \subset \H.$

With the proper definition of sheaves (Section \ref{sheaves}) and  cohomology, namely, the cohomology with respect to {\em axially symmetric coverings} (Definition \ref{asop}),
we can formulate the first main theorem of this paper:

\begin{theorem}[Main Theorem 1]\label{H1*}
  Let $\Omega$ be an axially symmetric domain. Then $$H^1(\Omega, \mathcal{SR}^*) = 0.$$
\end{theorem}
In quaternionic settings we have two types of jets, usual ones at a point, i.e. the value and some derivatives at a point are given, and spherical jets, given by finite spherical expansions (see \cite{GSS}).  Theorem \ref{H1*} yields the second main theorem of this paper,
\begin{theorem}[Main Theorem 2]\label{jets1}
  Let $\Omega$ be an axially symmetric domain and $Z$ a discrete set of points and spheres.
  Then there exists $f \in \mathcal {SR}(\Omega)$ with prescribed finite jets at $Z.$
\end{theorem}

In particular, on any axially symmetric domain, there exists  functions with prescribed zeroes  (Theorem  \ref{zeromult}).
Another consequence of Theorem \ref{H1*} is the existence of principal divisors, i.e. functions with prescribed zeroes and poles (Theorem \ref{divisor}). \\

The main method applied in the proofs of these theorems is gluing  local solutions. To this end we prove  additive and multiplicative Cartan's Splitting lemmas for various classes of slice--regular functions and use them for solving first and second Cousin problems for the aforementioned classes of slice--regular functions (Section \ref{cousin}, Theorems \ref{H^1SRR}, \ref{H^1SRR+}). To deal with the  Cousin problems one has to work with {\em Cartan coverings  }(Section 4 in \cite{PV2}).
In the last section, Section \ref{applications}, we apply our results to prove theorems, which are analogous to results on divisors and jet interpolation theorems on Stein manifolds.

\section{Preliminary results}\label{prelim}


 We denote by $\mathbb{H}$ the algebra of quaternions. Let
 $\mathbb{S}$ be the sphere of imaginary units, i.e. the set of
 quaternions $I$ such that $I^2=-1$.  Given any quaternion $z \not\in
 \mathbb{R},$ there exist (and are uniquely determined) an imaginary
 unit $I$, and two real numbers $x,y$ (with $y\geq 0$) such that
 $z=x+Iy$. With this notation, the conjugate of $z$ will be $\bar z :=
 x-Iy$ and $|z|^2=z\bar z=\bar z z=x^2+y^2$.
 Each imaginary unit $I$ generates (as a real algebra) a copy
 of a complex plane denoted by $\mathbb{C}_I:=\R+I\R$. We call such a complex
 plane a {\em slice}.  The upper half-plane in $\mathbb{C}_I$, namely $\{x+yI\ :\ y>0\}$ will be
 denoted by $\C_I^+$ and called a {\em leaf}. Set $\C_I^-:=\C_{-I}^+$ and for  a subset $E \subset \H$ define
 $E_I:=E \cap  \C_I, $  $E_I^+:=E \cap  \C_I^+, $ $E_I^-:=E \cap  \C_I^-, $ $E_I^0:=E \cap  \R.$
 We set $\mathbb{S}(a,r):=\{a + rI, I\in \mathbb{S}\} = a + r \mathbb{S}$ for $ a,r \in \R$.
The
 \emph{(axial) symmetrization} $\widetilde E$ of a subset $E$ of $\bH$ is defined by
 $$\widetilde{E} = \{ x + I y:  x,y \in \R, I\in \bS, (x + \bS y) \cap E \neq \varnothing \}.$$
 A subset $\Omega$ of $\bH$ is called
 {\em (axially) symmetric} (in $\bH$) if $\widetilde{\Omega} = \Omega.$  A subset $U$ of $\C$ is called
 {\em  symmetric} if it is invariant for the reflection $R$ over the real axis, i.e.  $U = R(U)$.



In the sequel we briefly recall some standard definitions (see \cite{GSS}, \cite{GPV}).

Natural domains of definition for slice--regular functions are {\em slice domains}.
A domain $\Omega$ of $\mathbb H$ is called a \emph{slice domain} if, for all $I\in \mathbb S$, the subset
$\Omega_I$ is a domain in $\C_I$ and $\Omega^0=\Omega\cap\R\neq \varnothing$.  If, moreover, $\Omega$ is axially symmetric, then it is called a {\em symmetric} slice domain.
Slice functions (see [GP]) are naturally defined also on {\em product domains}, i.e. 
axially symmetric domains of $\mathbb H \setminus \R$. 
Hence an axially symmetric domain $\Omega$ is either a symmetric slice domain or it is a product domain.


The present paper focuses on axially symmetric domains only and slice--regular functions defined on such sets. Contrary to the complex case, where a ball is a natural local domain of definition for holomorphic functions, in the case of slice--regular functions the natural domains, called {\em basic domains} (\cite{GPV}) are symmetrizations of topological discs in $\mathbb{C}_I \setminus \R$ ({\em basic product domains}) or symmetrizations of  symmetric topological discs in $\C_I$ ({\em basic slice domains}), $I \in \mathbb{S}$.
A basic domain is also a \emph{basic neighborhood} of any of its points.  We define also the empty set to be a basic set.
 An axially symmetric closed set  $V$ of $\mathbb H$ is called a {\em closed basic set}  if it is either a symmetrization of a closed topological discs in $\mathbb{C}_I \setminus \R$ or a symmetrization of a symmetric closed topological  disc in $\C_I$ for some $I \in \mathbb{S}$.

Notice that the intersection of a basic domain with the real axis is either empty or connected. The closure of a basic set is not necessarily a closed basic set. A closed basic set with a piecewise smooth boundary has a basis of basic sets.
 \begin{definition}[Symmetric open covering]\label{asop} Let $U \subset \C$ be a symmetric domain and
  let $\mathcal {U}=\{U_{\lambda}\}_{\lambda \in\Lambda}$
  be an open
covering of $U.$ The covering $\mathcal{U}$
is  a
{\em  symmetric open
  covering} if each $U_{\lambda} \in {\mathcal U}$ is a
symmetric  open set.
\end{definition}
\begin{definition}[Axially symmetric open covering]\label{asop} Let $\Omega \subset \H$ be an axially symmetric domain and
  let $\mathcal {U}=\{U_{\lambda}\}_{\lambda \in\Lambda}$
  be an open
covering of $\Omega.$ The covering $\mathcal{U}$
is  an
{\em axially symmetric open
  covering} if each $U_{\lambda} \in {\mathcal U}$ is an axially
symmetric  open set.
\end{definition}

If $\i^2 = -1,$
consider the complexification $\H_\C =\H+\i \H,$ of the skew field
$\H$ and set $\zeta\colon z = x+\i y \mapsto x-\i y=: \bar{z}$ to be the natural involution of
$\H_\C$. For any $J\in \bS$, define the map
\[\phi_J\colon \H_\C \to \H, \quad \phi_J(x+\i y)=x+Jy\]
Notice that the map $\phi_J$, when restricted to $\R_\C\cong \R+\i \R\cong \C$, is an isomorphism between $\C$ and $\R+J\R=\C_J$.

If $\Omega \subseteq \H$ is an axially symmetric domain, and if $i$
denotes the imaginary unit of $\C\subset\H$, then the intersection
$\Omega_{i}=\Omega\cap (\R+i\R)=\Omega\cap \C$ defines a domain of the
complex plane that is invariant under complex conjugation. 
With respect to the
established notations, the subset $\Omega_{\i}=\{x+\i y \in \H+\i \H :
x+iy \in \Omega_i\}$ is called the \emph{image of $\Omega_i$ in
  $\H_\C$}, 
  and is invariant under involution, i.e.,
$\Omega_{\i}=\zeta(\Omega_{\i})$.  We are now in a position to
recall the following definitions.

\begin{definition}[Stem function]\label{stem} Let $\Omega \subseteq \H$ be an axially symmetric open set, let $\Omega_{i}=\Omega\cap (\R+i\R)$ and let $\Omega_{\i}$ be the image of $\Omega_i$ in $\H_\C$.

A (continuous) function $F \colon \Omega_{\i} \to\H_\C$ is called a \emph{stem function} if $F(\bar z) = \overline{F(z)}$ for all $z \in \Omega_{\i}$. For each stem function $F \colon \Omega_{\i} \to \H_\C$, there exists a unique $f \colon \Omega \to \H$ such that the diagram
$$\begin{tikzcd}
\Omega_{\i}\arrow{d}{\phi_J}\arrow{r}{F}&\H_\C\arrow{d}{\phi_J}\\
\Omega\arrow{r}{f}&\H
\end{tikzcd}$$
commutes for all $J \in \bS$. The function $f$ is called the \emph{slice function} induced by $F$ and denoted by $\mathcal{J}(F)$.
The set of all slice functions on $\Omega$ is denoted by $\mathcal S(\Omega).$

Let $f=\mathcal J(F ), g=\mathcal J (G)$ be the slice functions induced by the stem functions $F, G$ respectively. The \emph{$*$-product} of $f$ and $g$ is defined as the slice function $f*g:=\mathcal J(FG)$.

\end{definition}
We will use  definitions of slice--regularity and of $*$-product that
involve stem functions, and that are valid for any axially symmetric
domain of $\H$. When restricted to symmetric slice domains, slice--regularity
coincides with the Definition 1.2 of slice--regularity initially
presented in \cite{GS}.

\begin{definition}[Slice--regular function]\label{regularitygp} Let $\Omega \subseteq \H$ be an axially symmetric open set.
A slice function $f \colon \Omega \to \H$, induced by a stem function $F \colon \Omega_{\i} \to \H_\C$, is called \emph{slice--regular} if $F$ is holomorphic.
The set of all slice--regular functions on $\Omega$ is denoted by $\mathcal{SR}(\Omega).$

A  slice
function $f\colon \Omega \to \H$ is said to be {\em slice--preserving} if and only if
$\forall I \in \bS, \forall z\in\Omega_I:=\Omega\cap \mathbb{C}_I$ we have that
$f(z)\in \C_I$. The set of all slice--regular
functions, which are slice--preserving in $\Omega,$ will be denoted as
$\cS\mathcal{R}_{\mathbb{R}}(\Omega)$.
The set  $\mathcal{SR}^*_{\R}(\Omega) \subset \mathcal{SR}_{\R}(\Omega)$ denotes the subset set of all nonvanishing functions and
 $\mathcal{SR}^+_{\R}(\Omega)\subset \mathcal{SR}_{\R}(\Omega)$ denotes the set of all functions  $\mathcal{SR}_{\R}(\Omega)$ which are strictly positive on the real axis.

 The sup-norm for a bounded function $f$ on $\Omega$ is denoted by $|f|_{\Omega}$.
\end{definition}

\begin{definition}[\cite{GSS}, Definition 5.21] A function $f$ is {\em semiregular} in a symmetric slice domain $\Omega$ if it is
regular in a symmetric slice domain $\Omega'\subset \Omega$
 such that every point of $\Omega \setminus \Omega'$ is
a pole (or a removable singularity) for $f$. We denote the set of semiregular functions on $\Omega$ by $\mathcal{SMR}(\Omega).$
\end{definition}

\begin{remark} The (conceptual) correspondence between the classes $\mathcal C(\Omega),$ $\mathcal O(\Omega)$ and $\mathcal M(\Omega)$ of (complex) continuous, holomorphic and meromorphic functions and the classes $\mathcal S(\Omega),$ $\mathcal{SR}(\Omega)$ and $\mathcal{SMR}(\Omega)$ of (quaternionic) slice, slice--regular and semiregular functions on $\Omega$ is presented below.
\[
  \mathcal C(\Omega)\leftrightarrow \mathcal S(\Omega) \qquad \mathcal O(\Omega)\leftrightarrow \mathcal{SR}(\Omega) \qquad \mathcal M(\Omega)\leftrightarrow \mathcal{SMR}(\Omega).
\]
\end{remark}

As shown in \cite{GSS}, the zero set of any $f \in \mathcal{SR}(\Omega), \ f \ne 0$, is a discrete set of points and spheres of the form $S(a,r)$ contained in $\Omega$.
 At each point $q_0 \in \Omega$, by Theorem 2.12 in \cite{GSS}, the function $f$ has a Taylor series expansion of the form
\[ f(q) = \sum_{n = 0}^{\infty} (q - q_0)^{*n} f^{(n)}(q_0)/n!,\]
which converges on some $\sigma$-ball $\Sigma(q_0,R)$.
The notion of a jet of $f$ of order $l$  at a point $q_0$  then consists of a sequence of  numbers $A_0, A_1,\ldots A_l$ representing the values $f^{(n)}(q_0)/n!= A_n, n = 0,\ldots,l$.

By Theorem 8.7 in \cite{GSS}, the function also has a {\em  spherical  expansion at a point $q_0 \in x_0 +y_0\bS$} of the form
\[
  f(q) = \sum_{n = 0}^{\infty}[(q - x_0)^2 + y_0^2]^{n}[A_{2n} + (q-q_0) A_{2n+1}],
\]
which converges on a Cassini ball $U(x_0 + y_0\bS,R)=\{q \in \mathbb{H}, |(q - x_0)^2 + y_0^2| < R^2\} \subset \Omega$.
We define a {\em spherical jet of order $2l+1$ at $q_0$} to be the sequence $A_0,\ldots,A_{2l+1}$, representing the finite part of spherical expansion. If $A_{2m}$ or $A_{2m+1}$ are the first nonzero coefficients in the above series expansion, then $2m$ is the spherical  multiplicity
of $f$ at $x_0 + y_0\bS$.

If the point $q_0$ on the sphere is not given, then we define {\em a spherical jet of order $2l+1$ on the sphere}  $S(x_0, y_0)$  to be a sequence $A_0,A_1,\ldots A_{2l}, A_{2l+1}$ with $A_{k} = 0$ for odd $k$.

A spherical interpolation theorem (with jets of order $0$) with approximation for entire slice-regular functions on $\H$ was proved in \cite{PV1}.

The next proposition recalls two well known technical results that will be extensively used in the sequel (see, e.g., \cite{GP}).


\begin{proposition}\label{regularitygp1} Let $\Omega \subseteq \H$ be an axially symmetric open set, and let $f, g\in \mathcal{SR}(\Omega)$ be two slice--regular functions. Then
\begin{enumerate}
\item[(a)] the $*$-product $f*g$ is a slice--regular function on $\Omega$;
\item[(b)] if $f$ is slice--preserving, then $f*g=fg=g*f$, i.e, the $*$-product coincides with the pointwise product;
\item[(c)] if $f(q_0) = 0,$ then  $(f*g)(q_0)=0$.
\end{enumerate}
\end{proposition}

The \emph{imaginary unit function}
 $\cI \colon \H \setminus \R \to \mathbb{S}$
is defined by setting $\cI(q) = I$ if $q \in \mathbb{C}_I^+.$  The function $\cI$ is slice--regular and slice--preserving, but it is not an open mapping and  it is not defined on any slice domain.\\

Consider now an axially symmetric open set  
$\Omega$  and $f \in \mathcal{SR}(\Omega).$ We have already defined the splitting $f = f_0 + f_v,$ where the scalar part
$f_0$ of $f$ is a slice--preserving function.
\begin{definition}
  The function $f\in \mathcal{SR}(\Omega)$ is a {\em vectorial function} if $f = f_v.$
  The set of  vectorial functions on $\Omega$ will be denoted by
  $\cS\mathcal{R}^v(\Omega)$. We have $\mathcal{SR}(\Omega) = \mathcal{SR}_{\R}(\Omega)\oplus\mathcal{SR}^v(\Omega).$
  \end{definition}


 In \cite{GPV} the notion of a vectorial class was introduced (compare Definitions 5.1, 5.2 and 5.3 therein). Here we extend the existing definition.
\begin{definition}\label{vec_class}
  Let $f_v \in \cS\mathcal{R}^v(U)$ and $g_v \in \cS\mathcal{R}^v(U'),$ where $U,U'\subset \H$
are axially symmetric domains  in $\mathbb{H}$ such that $U\cap U'\neq\varnothing$.
Take $\widetilde{p}\subset U\cap U'$; we say that $f_v$ and $g_v$ are \emph{equivalent at $\widetilde{p}$}, in symbols
 $f_v \sim_{\widetilde{p}} g_v$, if there exist an  axially symmetric neighborhood $V_{\widetilde{p}}$ of $\widetilde{p}$, $V_{\widetilde{p}} \subset U\cap U'$, such that  $f_v$ and $g_v$
are linearly dependent over $\cS\mathcal{R}_{\mathbb{R}}(V_{\widetilde{p}})$ in $V_{\widetilde{p}}$. We will denote by $[f_v]_{\widetilde{p}}$ the $\sim_{\widetilde{p}}$ equivalence class whose representative is $f_v$.
\end{definition}

Loosely speaking, a global vectorial class $\omega$ on an open axially symmetric set $\Omega$ is determined by choosing a set of  vectorial functions on an open axially symmetric covering $\mathcal{U} = \{U_{\lambda}\}_{\lambda \in \Lambda}$ of $\Omega$, namely the set
$\mathcal{V}=\{ f_{\lambda} \in \mathcal{SR}^v(U_{\lambda}), U_{\lambda} \in \mathcal{U} \}$, so that if $f_{\lambda}, f_{\mu} \in \mathcal{V}$, then  $[f_{\lambda}]= [f_{\mu}]$ on $U_{\lambda}\cap U_{\mu}$. 
By Remark 4.2 in \cite{GPV}, for each vectorial class $\omega$, we can choose a local representative $f_v$ having neither real nor spherical zeroes. Such representative is called a {\em minimal representative of $\omega$}. We  denote
$\mathcal{SR}_{\omega}(\Omega) := \{ f \in \mathcal{SR}(\Omega), [f_v] = \omega\} .$
Notice that $\mathcal{SR}_{\R} = \mathcal{SR}_{[0]}$.

A direct calculation shows that for $f,g \in \mathcal{SR}_{\omega}(U),$ we have $f*g = g*f$ (see also \cite{AdF1}).

We will prove later (Theorem \ref{GMR}), that every vectorial class on an axially symmetric open set $\Omega$ has a global minimal representative.

\begin{example}
  The class of slice--regular functions on an axially symmetric domain $\Omega$  which preserve the slice $\C_i$ coincides with $\mathcal{SR}_{[i]}(\Omega)$ because each function, which preserves the slice $\C_i$, is of the form $f = f_0 + f_1 i$  and thus defines an embedding of the set of holomorphic functions on symmetric domains into the set of slice--regular functions on axially symmetric domains.
\end{example}

Given a standard basis of $\H,$ the vectorial part of a slice--regular function can be decomposed further
(\cite{GMP}, \cite{CGS}, \cite{AdF1}):
indeed, the next result holds for slice functions in general (see \cite{GMP}).
\begin{proposition}\label{frm}
  Let $\{1, i, j, k\}$
  be the standard basis of $\H$ and assume $\Omega$ is an  axially symmetric domain of $\H$.
  Then the map
\[ (\cS\mathcal{R}_{\mathbb{R}} (\Omega))^4\ni  (f_0,\ f_1,\ f_2,\ f_3 ) \mapsto f_0 + f_1 i + f_2 j + f_3 k \in \cS\mathcal{R}(\Omega)\]
is bijective.
\end{proposition}

In the sequel, all bases of $\mathbb{H}\cong \mathbb{R}^4$ will be orthonormal (and positively oriented)
with respect to the standard scalar product of $\mathbb{R}^4$.
Proposition \ref{frm} implies that, given any $f, g \in \mathcal{SR}(\Omega)$, there exist unique
$f_l, g_l \in \cS\mathcal{R}_{\mathbb{R}}(\Omega), l = 0,\ldots,3$, such that
\begin{eqnarray*}
f = f_0 + f_1 i + f_2 j + f_3 k=f_0+f_v\\
g=g_0+g_1i+g_2j+g_3k=g_0+g_v
\end{eqnarray*}
With the above given notation, if we call \emph{regular conjugate} of $f$ the function
$f^c=f_0-f_v$ then we have
\begin{equation}\label{e1}
f_0= \frac{f+f^c}{2}.
\end{equation}
Furthermore, using Definition~\ref{stem} and Proposition~\ref{regularitygp}, we obtain the following expression for the $*$-product of $f$ and $g$:
\begin{equation}\label{*product}
  f*g = f_0g_0-f_1g_1-f_2g_2-f_3g_3+
  f_0g_v+g_0f_v+ \dfrac{f_v*g_v-g_v*f_v}{2}
\end{equation}
or, in terms of components with respect to the basis  $\{1, i, j, k\}$,
\begin{equation}\label{**}
  \begin{array}{ccc}
    (f*g)_0&=& f_0g_0-f_1g_1-f_2g_2-f_3g_3\\
    (f*g)_1&=& f_0g_1+f_1g_0+f_2g_3-f_3g_2\\
    (f*g)_2&=& f_0g_2-f_1g_3+f_2g_0+f_3g_1\\
    (f*g)_3&=& f_0g_3+f_1g_2-f_2g_1+f_3g_0F
    \end{array}
\end{equation}
%
We now set
\[f^s:=f_0^2+f_1^2+f_2^2+f_3^2=f*f^c=f^c*f\]
and call $f^s$ the \emph{symmetrization} of $f$.  Let us mention that
the equality (\ref{e1}) immediately implies $|f_0(z)| \leq |f(z)|$ and
replacing $f$ by $f * \iota$ for $\iota = i,j,k$ gives also the
remaining inequalities $|f_l(z) \leq |f(z)|,\, l = 1,2,3.$\\

It is known that Runge theorem holds for slice--preserving functions and for slice--regular functions.
In order to prove the Runge theorem also for $\mathcal{SR}_{\omega}(\Omega)$ (Theorem \ref{RungeV}), we will not only use  Theorem \ref{GMR}, but also the vanishing of $H^1(\Omega,\mathcal{SR}_{\R}^*)$ (Theorem \ref{H^1SRR} (1)) which follows from the vanishing of $H^1(\Omega,\mathcal{SR}_{\R})$ (Theorem \ref{H^1SRR+} (1)).
\begin{theorem}[Runge Theorem] \label{RungeV}
     Let $U\subset \Omega$ be axially symmetric domains, the set $U_I$ Runge in $\Omega_I$ for some $I \in \mathbb S$ and $\omega$  a vectorial class on $\Omega$.
    Then Runge theorem holds in classes (1) $\mathcal{SR}_{\R}(\Omega),$ (2) $\mathcal{SR}_{\omega}(\Omega)$ and (3) for general slice--regular functions.
\end{theorem}
The proofs of (1) and (3) follow from classical complex analysis results.
Runge theorem holds also for nonvanishing slice--regular functions.
\begin{theorem}[Runge theorem for $\mathcal{SR}^*$, Theorem 3.8 \cite{PV2}]\label{RungeM}
Let $K$ be a union of finitely many disjoint closed basic sets in $\H$ and let $f\in \mathcal{SR}^*(\Omega)$ with $\Omega$ an open axially symmetric neighborhood of $K$.
Then, for any $\varepsilon>0$, there exists $g\in {\mathcal{SR}^*}(\H)$ such that
\[ |f-g|_{K}<\varepsilon.\]
\end{theorem}
\begin{remark} Any finite  union $K$ of disjoint closed basic sets is Runge in any axially symmetric open set $\Omega \supset K$.
\end{remark}

\begin{remark} Runge Theorem \ref{RungeM} holds also if  the group $\mathcal{SR}^*(\Omega)$
is replaced by some other closed and connected subgroup, where Runge theorem is available, for example
the  subgroup  $\mathcal{SR}^*_{\omega}(\Omega).$
\end{remark}

\subsection{Matrix representations}\label{MatRep}

As complex numbers can be represented by matrices, so can  be quaternions.
There are at least two ways of representing quaternions as matrices in
such a way that quaternion addition and multiplication correspond to
matrix addition and matrix multiplication. One is to use $2\times 2$ complex
matrices, and the other is to use $4 \times  4$ real matrices.

In the terminology of abstract algebra, these are
injective homomorphisms from $ \mathbb {H} $
to the matrix rings $M(2,\mathbb{C})$ and $M(4,\mathbb{R})$, respectively.
Furthermore{\color{red},} there is a strong relation between quaternion units and Pauli matrices
$\sigma_1, \sigma_2, \sigma_3$
(which are  $2 \times 2$ complex Hermitian unitary matrices and appear in quantum mechanics)
since the real linear span of $\{\Id_2,\ i\sigma_1,\ i\sigma_2,\ i\sigma_3\}$ is isomorphic to the real algebra of quaternions.

Our aim is to provide any slice function with a matrix representation
whose entries are slice-preserving functions in such a way that the
$*$--product of slice functions will correspond to the matrix product of corresponding functions in the same spirit as \cite{AdF3}.

To this end we identify

\[ \mathbf{1}=\Id_4=\begin{bmatrix}1&0&0&0\\0&1&0&0\\0&0&1&0\\0&0&0&1\end{bmatrix} \leftrightarrow 1
\quad\mathbf{I}=\begin{bmatrix}0&-1&0&0\\1&0&0&0\\0&0&0&-1\\0&0&1&0\end{bmatrix} \leftrightarrow i\]

\[\mathbf{J}=\begin{bmatrix}0&0&-1&0\\0&0&0&1\\1&0&0&0\\0&-1&0&0\end{bmatrix}\leftrightarrow j
\quad\mathbf{K}=\begin{bmatrix}0&0&0&-1\\0&0&-1&0\\0&1&0&0\\1&0&0&0\end{bmatrix} \leftrightarrow k\]
since $\mathbf{I}^2=\mathbf{J}^2=\mathbf{K}^2=-\mathbf{1}$ and $\mathbf{I}\cdot \mathbf{J}=\mathbf{K}$ and so
\[f\leftrightarrow f_0 \mathbf{1}+f_1\mathbf{I}+f_2\mathbf{J}+f_3\mathbf{K}
\ \mathrm{or}\
f\leftrightarrow
\begin{bmatrix} f_0 &-f_1 & -f_2& -f_3\\
  f_1& f_0       &-f_3     &f_2           \\
  f_2&  f_3  & f_0 &  -f_1  \\
  f_3&  -f_2  &   f_1  & f_0
\end{bmatrix}=: M_f  .
\]
Thus, if analogously $g\leftrightarrow M_g=\begin{bmatrix} g_0 &-g_1 & -g_2& -g_3\\
  g_1& g_0       &-g_3     &g_2           \\
  g_2&  g_3  & g_0 &  -g_1  \\
  g_3&  -g_2  &   g_1  & g_0
\end{bmatrix} $, then

\[f*g\leftrightarrow M_{f*g}= M_f\cdot M_g\]
as desired (see also (\ref{**})).
Notice furthermore that $f^c\leftrightarrow M_f^T$,
where $M_f^T$ is the transpose of $M_f$, and so

\[ M_{f^s}=M_{f*f^c}=M_f\cdot M_f^T=\begin{bmatrix}
f^s&0&0&0\\
0&f^s&0&0\\
0&0&f^s&0\\
0&0&0&f^s
\end{bmatrix}=f^s \mathbf{1}.
\]
Moreover, if $f^s\not\equiv 0$, then $f^{-*}:= f^c/f^s$ is such that $f^{-*}*f=f*f^{-*}=\Id$, that is
\[M_{f^{-*}}\cdot M_f= M_f\cdot M_{f^{-*}}= \mathbf{1}.\]
 Because $\det M_f = (f^s)^2$ we  observe, as expected,
that $f$ has a $*$-inverse if and only if $f^s$ is not vanishing, i.e. $M_f$ is an invertible (or nonsingular) matrix.

Let $\iota \in \mathbb S$ 
and $\C_{\iota}$ the induced complex plane. Define $\mathcal M_{\C_{\iota}}$ to be the set of matrices
\[\mathcal M_{\C_{\iota}}= \{M=q_0 \mathbf{1}+q_1\mathbf{I}+q_2\mathbf{J}+q_3\mathbf{K}, q_0,\ldots,q_3 \in \C_{\iota} \}\] and
equip it with 
the operator norm 
$|A|_{\infty}:= \max_{i=0,..,3} \sum_{j=0}^3 |a_{ij}|.$

Because for our class of matrices the sums of absolute values of rows are the same, the induced norm satisfies
\[ |M| = \sum_{j=0}^3 |q_{j}| = |M^T| = |M|_{\infty} = |M|_1.\]
Notice that
$\mathcal M_{\R}= \{r_0 \mathbf{1}+r_1\mathbf{I}+r_2\mathbf{J}+r_3\mathbf{K}, r_0,\ldots,r_3 \in \R\} \simeq\H.$

\begin{definition} The set of all matrices $M_f$ with entries
slice--preserving slice functions defined in $\Omega\subseteq \H$ is
denoted by $\mathcal{M}_{\mathcal{S}_{\R}}(\Omega)$ and the subgroup of
non singular ones by $\mathcal{GL}_{\mathcal{S}_{\R}}(\Omega)$.
The subgroup of matrices with slice--regular slice--preserving functions is denoted by $\mathcal{M}_{\mathcal{SR}_{\R}}(\Omega)$ and the subgroup of invertible ones by $\mathcal{GL}_{\mathcal{SR}_{\R}}(\Omega).$
\end{definition}


Notice that if $z\in\Omega_I \setminus \R$, then  $M_f(z)\in\mathcal{M}_{\C_I}$ and for $z \in \Omega \cap \R$ we have $M_f(z)\in\mathcal{M}_{\R.}$
The inequalities
  $|f(z)| \leq |M_f(z)| \leq 4 |f(z)| $ hold and in addition $|M_f(z)| = |M_{f^c}(z)|.$
We equip the  set $\mathcal{M}_{\mathcal{S}_{\R}}(\Omega)$
  with the sup norm
\[ |M_f|_{\Omega}:=  \sup_{z\in\Omega}|M_f(z)|.\]
 The subaditivity, submultiplicativity and homogeneity of this norm are evident.

As customary, one puts, $M_f^n:=M_f\cdot M_f^{n-1}$ with $M_f^0:=\mathbf{1}$; therefore,
one can define
\[
\exp (M_f):= \sum_{n=0}^{+\infty} \dfrac{M_f^n}{n!} \,\mbox{ and }\, \log({\mathbf 1}-M_f):
  =-\sum_{n=1}^{+\infty} \dfrac{M_f^n}{n}\]
and observe that
\[ \exp{M_f}\leftrightarrow \exp_*(f)\,\mbox{ and }\, \log({\mathbf 1}-M_f)\leftrightarrow \log_*(1-f)\]
as introduced in \cite{AdF1, AdF2, GPV}. Notice that 
%
%
%
%
the series for exponential function converges for any matrix $M_f$
whereas the series for logarithm
 converges in $\Omega$  for all matrices $M_f$ satisfying $|M_f|_{\Omega}<1$.
Since $\mathbf{1}$ commutes with $M_f^n$ (for any $n$), one easily deduces that
\[ \exp(\log (\mathbf{1}-M_f))=\mathbf{1}-M_f,\quad  \exp (M_f) \cdot \exp (-M_f)=\mathbf{1}.\]
The second equality proves that $\exp(M_f)$ is nonsingular for any matrix $M_f$.
From the first equality, if $|\mathbf{1}-M_f|<1$, then $M_f=\log (M_g)$ with $M_g=\exp(M_f)$.

 Assume that $f = f_0 + f_v = f_0 + f_1 v \in \mathcal{SR}_{[v]}(U)$ with $U$ an axially symmetric open set and $v^2 = -v^s = -1$, $f_0,f_1 \in \mathcal{SR}_{\R}(U)$. Then $v$ is a minimal representative of $[v]$ and  we can represent $f$ by a matrix
\[M_f:=\begin{bmatrix}f_0& -f_1\\
f_1& f_0
\end{bmatrix}.\]
Moreover,  $\det M_f = f^s,$ so the matrix is invertible if and only if $f$ is nonzero. If also $g \in \mathcal{SR}_{[v]}(U),$ then $M_{f*g} = M_fM_g = M_gM_f$, which implies that $\exp(M_f + M_g) = (\exp M_f) (\exp M_g)$.\\

\section{Sheaf theory for slice--regular functions}\label{sheaves}

The notion of a sheaf of slice--regular functions was already studied in the paper \cite{css}. Here we give a slightly different definition of sheaves.

Following Propositions $4.2$ and $4.3$ in \cite{css} we have
\begin{proposition}\label{pASP} The assignment $\{U,\cS\cR(U)\}_{U  = \widetilde U \subset \H} $ of rings of slice--regular functions over all axially symmetric open sets $U$ in $\H$
together with the families of morphisms
$\rho_{U,V}: \mathcal{SR}(U) \ra \mathcal{SR}(V)$
for every $V \subset U,$ $U,V$ axially symmetric,
induced by restrictions, is a presheaf. The presheaf $\{U,\cS\cR(U)\}_{U  = \widetilde U \subset \H}$ is a sheaf denoted by 
$\mathcal{SR}$.

The same holds true for the assignments $\{U,\mathcal{SMR}(U)\}_{U  = \widetilde U \subset \H}$ and $\{U,\cS(U)\}_{U  = \widetilde U \subset \H}$ over all $U$ axially symmetric open sets. The corresponding sheaves are denoted by $\mathcal{SMR}$ and $\mathcal{S}$ respectively.

\end{proposition}
 In this manner, the stalks of the sheaf are germs of functions defined on the elements of the set $\mathcal{H}:=\{\widetilde{z}, \, z \in \H\}$ and denoted by
$f_{*,\widetilde{z}}$. Notice that $\mathcal{H} = \H/\sim,$ where the relation is $z \sim w \Leftrightarrow \,\exists \,a, r \in \R \mbox{ so that } z,w \in S(a,r)$. Recall that the symmetrization of any real point $x$  is the set $\widetilde{x} = \{x\}$. By a slight abuse of notation, in such a case,  we simply write $x$.

\begin{remark} Let us explain what is the main difference between the definition
of sheaves
in \cite{css} and the one in Proposition \ref{pASP}.
 In \cite{css} the authors have supposed that the open sets are arbitrary, but the functions 
  considered
  were 
  assumed to be
   extendable to the symmetrization of the open set where they are defined.
The problem that we encountered with this approach was, that when  we have three non axially symmetric sets $U_1,U_2,U_3,$
such that $U_1\cap U_2\cap U_3:=U_{123} = \varnothing$ and such that their symmetrizations
$\tilde{U}_1,\tilde{U}_2,\tilde{U}_3,$ are the same set $U,$ then any
choice of functions $f_{lm}$ in $U_{lm}$ which extend to the
symmetrization of $U_{lm}$, satisfying  only the  condition $f_{lm} =
-f_{ml},$ $1\leq l< m\leq 3$,  forms a cocycle. While the cocycle condition
$f_{123}=f_{12}+f_{23}+f_{31}$
is trivial
on the sets $U_{123}=\varnothing$, this is no longer true so for the corresponding symmetrized
sets. Therefore the cocycles do not `symmetrize' to cocycles.

Let us make an explicit example. Assume
$U_0:= \{q \in \H, d(q,\bS) < 1/2\}. $ Let $C$ be the circle $i \cos t
+ j \sin t, t \in [0,2\pi].$ Divide the circle into arcs $C_l:= [(l-1)
  2\pi/2,l 2\pi/3], l = 1,2,3$ and let $S_l$ be tubular neighbourhoods
of $C_l$ in $\bS$ such that $S_1 \cap S_2 \cap S_3 = \varnothing.$ Since
the set $U_0$ does not intersect the real axis, the imaginary unit
function $\cI$ is well--defined on $U_0$ and hence the sets $U_l = \{q
\in U_0, \Im{q} \in S_l\}, l = 1,2,3$ form an open covering of
$U:=U_1 \cup U_2 \cup U_3$. By construction,  $U_{123} = \varnothing$. Moreover, the symmetrization of each of
the sets $U_1, U_1, U_3$ is $U_0.$
 Set $f_{lm} = 1,$ $f_{mm} = 0$
and $f_{lm} = -f_{ml}, 1 \leq l < m \leq 3$ and denote the symmetrizations by
$\tilde{f}_{lm}.$ Then $\tilde{f}_{12} + \tilde{f}_{23}+
\tilde{f}_{31} = 1 \ne 0.$ We claim that $H^1(U,\mathcal {SR}) \ne 0$ in the sense of \cite{css}.
If it were zero, then we would have functions $f_l$ on $U_l$ with $f_{lm}=f_m
- f_l = 1, l < m $ and hence, by the identity principle,
the same would hold for the symmetrized
functions $\tilde{f}_l,$ $\tilde{f}_{lm} = \tilde{f}_m - \tilde{f}_l =
1, l < m.$ But then
$$
 1 = \tilde{f}_{12} + \tilde{f}_{23}+ \tilde{f}_{31} = \tilde{f}_2 - \tilde{f}_1+ \tilde{f}_3 - \tilde{f}_2+ \tilde{f}_1 - \tilde{f}_3 =0,
$$
which is a contradiction.
To put it shortly, the symmetrization process and cocycle condition are not commutative operations.\\

With the definition we have adopted this problem does not occur. At
first glance the requirement for symmetric coverings seems to be a
restriction since the sets we are dealing with are not small in
diameter. On the other hand, the set of functions considered in
\cite{css} is also restricted by the requirement that they are extendable
to the symmetrized set. On the level of presheaves both definitions
are about the same set of functions. The difference occurs on the
level of cocycles.
\end{remark}


Recall that $\mathcal{SR}_{\R} = \mathcal{SR}_{[0]}$ and  $\mathcal{SR}_{\omega}$ is an
$\mathcal{SR}_{\mathbb{R}}$-module. The sheaves
$\mathcal{SR}_{\mathbb{R}}$ and $\mathcal{SR}_{\omega}$ are subsheaves of $\mathcal{SR}$. All these three sheaves are also sheaves of abelian groups for addition. Notice that $\mathcal{SR}_{\mathbb{R}} \subset
\mathcal{SR}_{\omega}\subset
\mathcal{SR}$ for  any $\omega.$

   In addition, we define
  $\mathcal{SR}_{\R}^{+}\subset \mathcal{SR}_{\R}$  to be the subset of germs $f_{\widetilde{z}} \in\mathcal{SR}_{\R}$  
   which
   at the point $z_0 \in \R$ satisfy the condition
\begin{equation}\label{cond1}
f_{*,z_0}(z_0) \in  (0,\infty).
\end{equation}
Also $\mathcal{SR}_{\R}^{+}$ is a subsheaf of $\mathcal{SR}_{\R}$.

As customary, by
  $\mathcal{SR}^*$ we denote the set of all nonvanishing germs $f_* \in \mathcal{SR} $  and define
    $\mathcal{SR}_{\R}^{*,+} :=\mathcal{SR}^+_{\R}\cap \mathcal{SR}^*  $,
    $\mathcal{SR}_{\R}^* :=\mathcal{SR}_{\R}\cap \mathcal{SR}^*  $,
    $\mathcal{SR}_{\omega}^*: = \mathcal{SR}_{\omega} \cap \mathcal{SR}^*.$  The sheaves  $\mathcal{SR}_{\R}^*$ and   $\mathcal{SR}_{\omega}^*$ are sheaves of abelian groups for $*$-multiplication.
We also set $\mathcal{SMR}^*:=\mathcal{SMR} \setminus\{0\}.$

\begin{definition}

In $\mathcal{SMR}$ two
germs $f_*, g_* \in \mathcal{SMR}$ are {\em equivalent}, $ f_*\sim g_*$, if there exists $h_* \in \mathcal{SR}^*$ so that $f_* = g_* * h_*.$

A {\em divisor} on an axially symmetric domain $\Omega$ is an element in $\Gamma(\Omega,\mathcal{SMR}^*/\sim).$
\end{definition}
Given a divisor on an axially symmetric domain $\Omega$, there exists a basic covering $ \{U_l\}_{l \in \bN} $
of $\Omega$
so that the divisor is represented by  the pairs $\{(f_{l}, U_{l}),\  f_l \in \mathcal{SMR}(U_l)\}_{l \in \bN}$ satisfying $f_{kl} = f_k^{-*}*f_l   \in \mathcal{SR}^*(U_{kl}), k,l \in \bN.$

\begin{remark}
The support of a divisor on an axially symmetric domain $\Omega$
 is a (discrete) set of points and spheres in $\Omega$.
\end{remark}

We recall the following (see \cite{GPV}).

\begin{proposition}\label{TTT} Let $\omega  = [v]$  be a  vectorial class on an axially symmetric domain $\Omega$  such that $v^s \ne 0.$  The exponential maps between sheaves induced by $\exp_*,$
\begin{eqnarray}
  &&\mathcal{E}_* : \mathcal{SR} \ra \mathcal{SR}^* \label{notkernel}\\
  &&\mathcal{E}_* : \mathcal{SR}_{\omega} \ra \mathcal{SR}_{\omega}^* \nonumber \\
  &&\mathcal{E}_* : \mathcal{SR}_{\R} \ra \mathcal{SR}_{\R}^{*,+} \nonumber
\end{eqnarray}
are surjective and the latter two are also   sheaf  homomorphisms.

 If $v$ is a minimal representative of $\omega$ on a basic neighbourhood $U$ of $z$ in $\Omega$, then the kernels ${\mathcal K}_{\omega}$ and ${\mathcal K}_{\R}$
have the following local description:
\begin{itemize}
  \item [i)]
  if  $z \not \in \R$ and $ v^s(z) \ne 0
       \Rightarrow {\mathcal K}_{\omega, z} =\{ (2k + \delta)\pi \cI + (2n + \delta)\pi \frac{v}{\sqrt{v^s}},\, k,n \in \Z, \delta = 0,1 \},$\\
\item[ii)] if $ z \in \R$  and $ v^s(z) \ne 0 \Rightarrow {\mathcal K}_{\omega, z} = \{  2n \pi \frac{v}{\sqrt{v^s}},\, n \in \Z\},$\\
  \item[iii)]if $ z \not \in \R$  and $v^s(z) = 0\Rightarrow {\mathcal K}_{\omega, z} = \{ 2k \pi \cI,\, k \in \Z\},$\\
   \item[iv)] if $z \not \in \R$, ${\mathcal K}_{\R,z} = \{ 2k \pi \cI, \,k \in \Z\}$    and   if $ z \in \R$, ${\mathcal K}_{\R,z} = \{ 0\}$.
 \end{itemize}
 Since $\mathcal{SR} = \cup_{\omega}\mathcal{SR}_{\omega}$ we have for each $z \in \Omega$
 $$
   {\mathcal K}_{z} := \mathcal{E}_*^{-1}(1) = \cup_{{\omega}}{\mathcal K}_{\omega,z},
 $$
 with $\mathcal{E}_*^{-1}(1)$ denoting the inverse image of $1$ in (\ref{notkernel}).
\end{proposition}

\begin{theorem}
The sequences
$$
  0 \ra {\mathcal K}_{\omega}  \ra  \mathcal{SR}_{\omega} \xrightarrow{\mathcal{E}_*} \mathcal{SR}_{\omega}^{*} \ra 0
$$
$$
  0 \ra {\mathcal K}_{\R}  \ra  \mathcal{SR}_{\R} \xrightarrow{\mathcal{E}_*} \mathcal{SR}_{\R}^{*,+} \ra 0
$$
are  exact.
\end{theorem}

\begin{remark}
The sequence
$$
  0 \ra {\mathcal K}  \ra  \mathcal{SR} \xrightarrow{\mathcal{E}_*} \mathcal{SR}^{*} \ra 0
$$
is not exact. If $f \in \mathcal{SR}_{\omega}(\Omega)$ and $g \in \mathcal{SR}_{\omega_1}(\Omega)$, $\omega \ne \omega_1,$ then
$\exp_*(f + g) \ne \exp_*f*\exp_* g$ in general.
In particular, if $f = 2\pi i,$ $g = 2\pi j$, then $\exp_*f=\exp_*g=1$ but
$$
  \exp_*(f+g) = \cos(2\sqrt{2}\pi ) + \frac{i+ j}{\sqrt 2}\sin(2\sqrt{2}\pi ) \ne 1.
$$
\end{remark}

We begin by studying the \v Cech cohomology of $\Omega$ with values in the abelian group
$\mathcal{I}\Z$ with respect to axially symmetric coverings.

\begin{definition}[]
Let $\Omega\subset \mathbb{H}$ be an axially symmetric domain in $\H$  and $\mathcal{U} = \{U_{\lambda}\}_{ \lambda \in \Lambda}$ an axially symmetric open  covering of $\Omega$.
{\em An $m$-cochain of  $\mathcal{U}$ with coefficients in $\mathcal{I}\mathbb{Z}$}  is any collection
$\{(f_L,U_L), f_L =\mathcal{I} n_L, \ n_L\in\Z\}_{L \in \Lambda^m}$ with $f_L$ continuous.
\end{definition}
Notice that on a slice domain the function $\mathcal{I} n_L$ is continuous if and only if $n_L = 0$.
Cochain complexes  and coboundary operators are defined as usual
 and corresponding cohomology groups denoted by $H^n(\Omega,\mathcal{I}\mathbb{Z})$.


\begin{theorem}\label{HNIZ}
Let $\Omega$ be an axially symmetric domain. Then 
\[ H^n(\Omega,\mathcal{I}\mathbb{Z}) = 0 \mbox{ for } n \geq 2.\]
\end{theorem}
We prove the theorem by restricting the data on a slice. This  induces {\em antisymmetric complexes} introduced in  \cite{PV2}.
\begin{definition}[Antisymmetric complexes and cohomology groups]
Let $D\subset \mathbb{C}$ be an open symmetric set and $\mathcal{U} = \{U_{\lambda}\}_{ \lambda \in \Lambda}$ an open symmetric covering of $D$.
{\em An antisymmetric $m$-cochain} of $\mathcal{U}$  is any collection $\{(f_L,U_L), f_L : U_L \rightarrow \mathbb{Z}\}_{L \in \Lambda^m}$ with $f_L$ continuous, constant in $U_L^+$ and  satisfying the antisymmetric property:
\[f_L(z) = - f_L(\bar{z}).\]
Antisymmetric cochain complexes  and coboundary operators are defined as usual
 and corresponding cohomology groups denoted by $H^n_{a}(\mathcal{U},\mathbb{Z})$.
 The open symmetric coverings of $D$ form a directed set under refinement, therefore we define $H^n_{a}(D,\mathbb{Z})$ to be a direct limit over symmetric open coverings of $D$.
\end{definition}
In \cite{PV2}, Theorem 5.2 we have proved that for any axially symmetric domain $\Omega \subset \H$ we have $H^n_{a}(\Omega_I,\mathbb{Z})=0$ for all $n \geq 2$.
\begin{proposition}\label{ASC}
Let $\Omega$ be an axially symmetric domain in $\H$. By restricting to $\Omega_I,$ any cochain  on $\Omega$ with values in $\mathcal{I} \mathbb{Z}$ induces an antisymmetric cochain on $\Omega_I$. Vice versa, by symmetrization of domains, every antisymmetric cochain on $\Omega_I$ induces a cochain   on $\Omega$ with values in $\mathcal{I} \mathbb{Z}$.
\end{proposition}
\begin{proof}


Consider an $m$-cochain defined on an axially symmetric open  covering  $\mathcal{U} = \{U_n\}_{n \in \bN_0}$ of $\Omega$ by $C=\{({\mathcal I} n_{L}, U_{L})\}_{ L \in \bN_0^m}$, where $n_{L} = 0$ if $U_{L} \cap \R = \varnothing$. Choose $I \in \mathbb{S}$ and consider the covering
 $\mathcal{U}_I = \{U_{n,I}\}_{n \in \bN_0}$ of the symmetric set $\Omega_I \subset \mathbb{C}_I$.
 Notice that if the open set $U_{k,I}$
intersects the real axis, it is a topological disc with
$\refl{U_{k,I}} = U_{k,I}$ and hence $n_{L} = 0$ for all $L$ that contain $k$. If it does not intersect the real
axis then $U_{k,I}$ consists of two symmetric topological discs
$U_{k,I}^+$ and $U_{k,I}^-.$
Hence $C$ equals $In_{L}$ on $U_{L}^+$ and $-In_{L}$ on $U_{L}^-$ and so assigns the value $n_{L}$ on $U_{L}^+$ and the value $-n_{L}$ on $U_{L}^-.$ The reverse and the last statement are  obvious.

\end{proof}

\begin{proof}[Proof of Theorem \ref{HNIZ}] The proof is a direct consequence of Proposition \ref{ASC} and the fact that $H^n_{a}(\Omega_I,\mathbb{Z})=0$ for all $n \geq 2$.
\end{proof}

\section{Cousin problems and Cartan's Splitting Lemmas for Classes of Slice Regular Functions }
\label{cousin}
This section is devoted to  the proofs of  the following two theorems and their applications.
\begin{theorem}\label{H^1SRR+} Let $\Omega$ be an axially symmetric domain in $\H$ and  $\omega$ a global vectorial class on $\Omega$. Then
\begin{itemize}
\item[1)] $H^1(\Omega, \mathcal{SR}_{\R}) = 0$;
\item[2)]  $H^1(\Omega, \mathcal{SR}_{\omega}) = 0$;
\item[3)] $  H^1(\Omega,\mathcal{SR})=0.$
\end{itemize}
This implies that  the first (or additive) Cousin problem is solvable in  $\mathcal{SR}_{\R}$, in $\mathcal{SR}_{\omega}$ and in  $ \mathcal{SR}$.
\end{theorem}

\begin{theorem}\label{H^1SRR} Let $\Omega$ be an axially symmetric domain in $\H$ and $\omega$ a global vectorial class on $\Omega$. Then
\begin{itemize}
\item[1)] $H^1(\Omega, \mathcal{SR}^{*,+}_{\R})= H^1(\Omega, \mathcal{SR}_{\R}^{*}) = 0$;
\item[2)]  $H^1(\Omega, \mathcal{SR}^{*}_{\omega}) = 0$;
\item[3)] $  H^1(\Omega,\mathcal{SR}^*)=0.$
\end{itemize}
This implies that  the second (or multiplicative) Cousin problem is solvable in  $\mathcal{SR}^*_{\R}$,  in $\mathcal{SR}^*_{\omega}$ and in  $ \mathcal{SR}^*$.
\end{theorem}
Theorem  \ref{H1*} is a special case of Theorem \ref{H^1SRR}.
The claim in Theorem \ref{H^1SRR+} (3) is analogous to the one of Proposition $4.8$ in \cite{css}.\\

To have a concise presentation, the analogous results for all the three classes of functions are stated in one theorem. However, the proofs of all the three claims are based on the proofs for the class of slice--preserving functions. In particular, Theorem \ref{H^1SRR+} (1)  implies Theorem \ref{H^1SRR+} (3) and
Theorem \ref{H^1SRR} (1). Theorem \ref{H^1SRR} (1) implies Theorem \ref{GMR} on existence of global minimal representatives of the vectorial class $\omega$ on an axially symmetric open set $\Omega$ and consequently the Runge theorem \ref{RungeV}
holds in $\mathcal{SR}_{\omega}$.
Theorem \ref{H^1SRR} (1) and  Theorem \ref{GMR} then imply Theorem \ref{H^1SRR+} (2).

\subsection{Additive and multiplicative Cartan's Splitting lemmas}
We begin  this subsection with the proofs of additive and multiplicative
Cartan's Splitting Lemmas for slice--preserving functions, functions with
prescribed vectorial classes and for general slice--regular functions. The reason why  these classes are treated separately is that the $*$-product is commutative in the first two classes.

{
We recall the notion of Cartan pairs and strings from  Section 4, \cite{PV2}.
If $A, B \subset  \bH$ are axially symmetric compact sets with piecewise smooth boundaries satisfying $\overline{(A \setminus B)} \cap \overline{(B \setminus A)} = \varnothing$, $B$  a closed basic set and $A \cap B$ a finite union of disjoint closed basic sets,
 then we say that $(A,B)$  form a {\em Cartan pair or a Cartan $2$-string}.  A sequence $A_1,\ldots, A_n \subset \bH$ of axially symmetric  compact sets with piecewise smooth boundary is a Cartan $n$-string if $(A_1,\ldots, A_{n-1})$ and $(A_1 \cap A_n,\ldots A_{n-1} \cap A_n)$ are Cartan $(n-1)$-strings and $(A_1\cup \ldots \cup A_{n-1}, A_n)$ is a Cartan pair.

  A sequence $\mathcal {A}=\{A_n\}_{n \in \bN_0}$  of closed basic sets is a {\em Cartan sequence} if  for each $n \in \bN$, then $(A_0,\ldots,A_{n})$ is a Cartan $n+1$-string and for distinct $n_1,n_2,n_3,n_4$, we have that $A_{n_1n_2n_3n_4} = \varnothing$ and
  $A_{n_1n_2n_3}$, $A_{n_1n_2}$  are  closed basic sets if nonempty;
    if $A_{n_i} \cap \R \ne \varnothing$ for distinct $n_1, n_2,n_3$, then  $A_{n_1n_2n_3} = \varnothing$ and
    the set $A_{n_1n_2}$ is a slice basic domain if nonempty.

 A covering $\mathcal{B} = \{U_k\}_{k \in \bN_0}$ is a {\em Cartan covering} of $\Omega$ if it is a covering of $\Omega$ and  the sequence $\overline{U}_0,\overline{U}_1,\ldots $ is a Cartan sequence.

  Theorem 1.1 in \cite{PV2}, states that given ${\mathcal U}$ a locally finite axially symmetric
  open covering of an axially symmetric domain $\Omega\subset \mathbb{H}$ and  $Z \subset
  \Omega$,  a discrete set of points or spheres, there exist a
  Cartan covering $\mathcal B = \{U_n\}_{ n \in \bN_0}$ 
  subordinated to $({\mathcal U},Z).$
  Even more, there exist a sequence $\{\ve_n\}_{n \in \bN_0}$ such that also the coverings
  \[ {\mathcal B}^t:= \{ U_n^t= U_n + B(0,t\ve_n)\}_{ n \in \bN_0} \]
  are Cartan and subordinated to  $({\mathcal U},Z)$ for all $t \in [0,1].$ In particular, $\overline{U}_n = \cap_{t > 0} U^t_n.$
}
\begin{lem}[Additive Cartan's splitting lemma  with estimates]\label{CSL0}
Let $(A,B)$ be a  Cartan pair, $C:= A \cap B$ and $U$ its open neighbourhood.
Let $U_A\Subset \H$ be a (topologically) regular  axially symmetric open neighbourhood of $A$,  $U_B\Subset \H$ a (topologically) regular  axially symmetric open
neighbourhood of $B$
so that $(\overline{U_A},\overline{U_B})$ is a Cartan pair with
$U_C:= U_A \cap U_B \Subset U.$

There exists a constant $D$,  so that  any  $\gamma \in \mathcal{SR}(U)$  can be decomposed as $\gamma = \alpha + \beta$ on $U_C$
with  $\alpha\in \mathcal{SR}(U_A),$ $\beta \in
\mathcal{SR}(U_B),$ satisfying
the estimate
\begin{equation}\label{estimate1}
   |\alpha|_A, |\beta|_B \leq D |\gamma|_{\overline{U_C}}.
\end{equation}
In addition,

(a) if  $\gamma$ is slice--preserving, then we can choose $\alpha$ and $\beta$ to be slice--preserving;

(b) if $\omega$ is a vectorial class on $\Omega \Supset (U \cup U_A \cup U_B)$, $g_v$ its global minimal representative on $\Omega$, $g_v \ne 0$ on $U$, and if $\gamma \in \mathcal{SR}_{\omega}(U),$ then we can choose $\alpha\in \mathcal{SR}_{\omega}(U_A)$ and $\beta \in
\mathcal{SR}_{\omega}(U_B).$
\end{lem}


\begin{proof} The proof is divided in three parts.

 {\bf Proof for $\mathcal{SR}_{\R}$.}
  We claim that there exist functions
$\alpha\in \mathcal{SR}_{\R}(U_A)$, $\beta\in \mathcal{SR}_{\R}(U_B),$  $\gamma = \alpha + \beta$ and such that the estimate (\ref{estimate1}) holds.
Fix $I \in \mathbb{S}$. By classical results (compare \cite{GR}, VI.E, Lemma 6), there exist holomorphic functions $\alpha_1 \in \mathcal{O}(U_{A,I})$ and
$\beta_1 \in \mathcal{O}(U_{B,I})$ so that  $\gamma_I = \alpha_1 + \beta_1$ satisfying the estimate $|\alpha_1|_A, |\beta_1|_B \leq D |\gamma_I|_{\overline{U_{C,I}}}$.
Define $\alpha_I$  by $\alpha_I(z) = (\alpha_1(z) + \overline{\alpha_1(\bar{z})})/2$ and $\beta_I$ by $\beta_I(z) = (\beta_1(z) + \overline{\beta_1(\bar{z})})/2$.
Then $|\alpha_I|_A, |\beta_I|_B \leq D |\gamma_I|_{\overline{U_{C,I}}}$.


By extension formula, the functions $\alpha_{I}$ and $\beta_{I}$ extend to slice--preserving functions $\alpha, \beta$. Then, by construction, the equality
$\gamma = \alpha + \beta$ and  the estimate (\ref{estimate1}) hold.

{\bf Proof for $\mathcal{SR}_{\omega}$.}  Set $M := |g_v|_{\overline{U_A \cup U_B}}$ and $m:=\min\{|g_v(q)|, q\in \overline{U_C}\}.$
Decompose the  function $ \gamma = \gamma_0 + \gamma_v.$ Since  $g_v$ is a minimal representative of the vectorial class $\omega$, we have
$\gamma_v =\gamma_1 g_v \mbox{ with } \gamma_1 \in \mathcal{SR}_{\R}$.
A trivial estimate  $|\gamma_1|_{\overline{U_C}} \leq |\gamma_v|_{\overline{U_C}}/m$ holds.
 For $l = 0,1,$ the functions $\gamma_l$ can be decomposed to $\gamma_l
= \alpha_l + \beta_l, $ with $\alpha_l \in \mathcal{SR}_{\R}({U}_A),$ $
\beta_l \in \mathcal{SR}_{\R}({U}_B)$ with the estimate
\begin{equation}\label{est1}
   |\alpha_l|_A, |\beta_l|_B \leq D_1 |\gamma_l|_{\overline{U_C}}.
\end{equation}
Define
$ \alpha:= \alpha_0 + \alpha_1 g_v,\quad \beta = \beta_0 +  \beta_1 g_v.$
Together with the decomposition
$\gamma = \alpha + \beta$
we also have the estimate
\begin{eqnarray*}
  |\alpha|_A &=& |\alpha_0 + \alpha_1 g_v|_A \leq |\alpha_0|_A + |\alpha_1|_A |g_v|_{A} \\
    &\leq& D_1 \left(|\gamma_0|_{\overline{U_C}} +|\gamma_1|_{\overline{U_C}}|g_v|_{A}\right)\\
   &\leq& D_1 \left(|\gamma_0|_{\overline{U_C}} +|\gamma_v|_{\overline{U_C}} \frac{|g_v|_{\overline{U_A \cup U_B}}}{m} \right)\\
    &\leq&  D_1 (1 +{M}/{m}) |\gamma|_{\overline{U_C}} = D|\gamma|_{\overline{U_C}};
\end{eqnarray*}
the last inequality follows from the definition of the decomposition $f = f_0 + f_v, $ since
$f_0(z)  =\frac12(f(z) + \overline{f(\bar{z})})$ and $f_v(z)  =\frac12(f(z) - \overline{f(\bar{z})})$
and so  on an axially symmetric compact set $K$ the estimates $|f_0|_K,|f_v|_K \leq |f|_K$ hold.
Similarly
$$
  |\beta|_B \leq D_1 (1 +{M}/{m})  |\gamma|_{\overline{U_C}};
$$

{\bf Proof for $\mathcal{SR}$.} Let $\{1, i, j, k\}$  be the standard basis of $\H$
and  $\gamma = \gamma_0 + \gamma_1 i + \gamma_2 j + \gamma_3 k \in \cS\mathcal{R}(U)$ the unique decomposition of $\gamma$ with $\gamma_l\in \cS\mathcal{R}_{\mathbb{R}} (\Omega), l= 0,\ldots,3.$ Since for each $l$ we have $ |\gamma_l|_{\overline{U_C}} \leq |\gamma|_{\overline{U_C}}$,  the case for slice--preserving functions immediately implies that there exist
$\alpha_l\in \mathcal{SR}(U_A),$ $\beta_l \in
\mathcal{SR}(U_B),$ such that $\gamma_l=\alpha_l + \beta_l$ and
the estimate $|\alpha_l|_A, |\beta_l|_B \leq D |\gamma_l|_{\overline{U_C}}$ holds and hence $|\alpha|_A, |\beta|_B \leq 4D |\gamma|_{\overline{U_C}}.$
\end{proof}
\begin{proof}[Proof of  Theorem \ref{H^1SRR+} (1) and (3)] 
The proof of the claim (3) follows immediately once we have proved claim (1).

To prove claim (1), let an element in $H^1(\Omega,\mathcal{SR}_{\R})$ be represented by a cocycle defined on a sufficiently
  fine  axially symmetric open covering $\mathcal U. $
  Let $\mathcal{B} = \{U_k, \, k \in \bN_0\}$ and
  $\mathcal{B}^{1/(j+1)} = \{U_k^{j}\}_{ k \in \bN_0} $ be Cartan coverings given by Theorem 1.1 in \cite{PV2} subordinated to $\mathcal{U}$, where the sets $\{U_k^j\}_{j \in \bN_0}$ are a basis of open neighbourhoods of $\overline{U_k}$. 
  Let  the cocycle be represented by the functions $v_{kl} \in \mathcal{SR}(U^0_{kl})$.
  Our aim is to find  functions $v_n$ defined on $U_n,$ $n \in \bN_0$ so that $v_{kl} = v_k - v_l$ on $U_{kl}.$
  Choose $\ve > 0$ and proceed by induction.
 of $U_{01}^1.$

$n = 0.$ Define $A:=\overline{U_0^2},$
 $B:= \overline{U_1^2}$, $U_A:=U_0^1$, $U_B:=U_1^1$ and $U = U_{01}^0$. Then the Cartan pair $(A,B)$ and sets $U_A, U_B, U$ fulfill the assumptions of Lemma \ref{CSL0}. Let $D$ be the constant given by this lemma. 
 Because the set $\overline{U_{01}^1}$ is a  closed basic  set, the function $c=v_{01}$ can be approximated on $\overline{U_{01}^1}$ by an entire  slice--regular slice preserving function $h$, $|c-h|_{U_{01}^1} < \ve/(2D)$.
Lemma \ref{CSL0} (a)  for the function $c - h$ defined on $U_{01}^0$ yields slice--preserving slice--regular
functions $a_1, b_1$ defined on  $U_A$ and $U_B$ respectively, so that $c = a_1 -b_1,$ $|a_1|_{A} <
\ve/2$. Set $v^1_0 = a_1$ on $U_0^1$ and $v^1_1 = b_1 - h$ on $U_1^1.$ Then
$v_{01} := v_0^1-v_{1}^1$ on a $U_{01}^1.$

$n \ra n+1.$ Assume that the functions $v^n_k \in
\mathcal{SR}_{\R}(U^{n}_{k})$, $k =0,\ldots,n,$ have already been
defined.  Set $A :=\overline{U_0^{n+2} \cup \ldots \cup U_n^{n+2}}$
and $B =\overline{U_{n+1}^{n+2}}$, $U_A =U_0^{n+1} \cup \ldots \cup U_n^{n+1}$, $U_B = U_{n+1}^{n+1}$ and $U = (U_0^n \cup \ldots \cup U_n^n) \cap U_{n+1}^n$. As above, the requirements of Lemma \ref{CSL0} are fulfilled for sets $A,B,U_A, U_B$ and $U$. Let $D$ be the constant yielded by this lemma. Consider the functions $v_{k\,n+1}$
with $k\leq n.$ The cocycle condition implies, that
the functions $v_{k\,n+1}- v_k^n$ defined on
$U_{k\, n+1}^{n},$ $k \leq n,$ coincide on the intersections of their
domains of definition,
\[v_{k\,n+1}- v_k^n=v_{l\,n+1}- v_l^n \Leftrightarrow v_{k\,n+1}- v_{l\,n+1}=v_k^n- v_l^n=v_{kl},\]
 and hence define a function $c$ on $U$ 
  By definition of Cartan pairs, the set $\overline{U_A} \cap \overline{U_B}$ is a  finite union of closed basic sets
  therefore there exists an entire function $h$ so that $|c-h|_{U_A \cap U_B} <\ve/(D 2^{n+1})$. By Lemma \ref{CSL0} (a), the function $c-h$ splits as $c - h = a_{n+1} -
b_{n+1}$ with functions $a_{n+1}, b_{n+1}$ slice--regular slice--preserving and defined on  $U_A$ and $U_B$ respectively
and  the estimate
 $|a_{n+1}|_A < \ve/2^{n+1}$  holds.  Set
$v_k^{n+1}:= v_k^n + a_{n+1}$ on $U_k^{n+1},$ $k \leq n$ and
$v_{n+1}^{n+1}:=b_{n+1}-h$ on $U_{n+1}^{n+1}.$
Obviously, $v_{k l} = v_k^{n+1}- v_l^{n+1}$ holds for $k,l\leq
n$. For $k \leq n$ we have
\[ v_k^{n+1} - v_{n+1}^{n+1} = v_k^n + a_{n+1}- b_{n+1} + h  = v_k^n + c = v_{k\,n+1}  \mbox{ on }  U_{k\,n+1}^{n+1} \]
hence $v_{kl} = v^{n+1}_k - v_l^{n+1}$ on $U_{kl}^{n+1}$ for all $k,l \leq n+1$.

For each $k\in \bN_0$  the sum $\sum_{n = k+1}^{\infty} a_n$ converges uniformly on $\overline{U_k},$ so
 $v_k:= v_k^k + \sum_{n = k+1}^{\infty} a_{n} = \lim_{n \ra \infty} v_k^n$ is well defined and $v_{kl} = v_k - v_l$.

\end{proof}

In order to investigate the solvability of the second (or multiplicative) Cousin problem, we start with the
corresponding multiplicative Cartan's Splitting
Lemmas for all the three types of classes of slice regular functions.
\begin{lem}[Multiplicative Cartan's splitting lemma with estimates]\label{MCSL0}
Let $(A,B)$ be a  Cartan pair, $C:= A \cap B$ and $U$ its open neighbourhood.
Let $U_A\Subset \H$ be a (topologically) regular axially symmetric open neighbourhood of $A$, $U_B\Subset \H$ a (topologically) regular axially symmetric open
neighbourhood of $B$ so that  $(\overline{U_A},\overline{U_B})$ is a Cartan pair with
$U_C:= U_A \cap U_B \Subset U.$

Let $\ve > 0.$ For any $c \in \mathcal{SR}^{*}(U)$  there exist $a \in
\mathcal{SR}^{*}(U_A)$  and  $b \in
\mathcal{SR}^{*}(U_B)$ such that $c = a * b $ on $U_C$ and $|a -
1|_{A} < \ve.$ In addition,

(a) if  $c\in \mathcal{SR}^{*,+}_{\R}(U)$  then we can choose $a\in \mathcal{SR}^{*,+}_{\R}(U_A)$ and $b\in \mathcal{SR}^{*,+}_{\R}(U_B)$; if, moreover,
$\overline{U_C}$ is a closed basic set, then
for any $c \in \mathcal{SR}^{*}_{\R}(U)$  we can choose  $a \in
\mathcal{SR}_{\R}^{*,+}(U_A)$ and  $b \in
\mathcal{SR}^{*}_{\R}(U_B)$;

(b) if $\omega$ is a vectorial class on $\Omega \Supset (U \cup U_A \cup U_B)$, $g_v$ its global minimal representative on $\Omega$, $g_v \ne 0$ on $U$, and if $\gamma \in \mathcal{SR}_{\omega}(U),$ then  we can choose   $\alpha\in \mathcal{SR}_{\omega}(U_A)$ and $\beta \in
\mathcal{SR}_{\omega}(U_B).$

\end{lem}

\begin{proof}
 Let ${U}'_A, {U}'_B \Subset \bH$ be two (topologically) regular axially symmetric open
  neighbourhoods of $U_A, U_B$ respectively, so that $(\overline{U'_A}, \overline{U'_B})$ is a Cartan pair
  and ${U}'_C := {U}'_A \cap {U}'_B\Subset U.$

{\bf Proof for $\mathcal{SR}_{\R}$.}  Because $(\overline{U'_A},\overline{U'_B})$ is a Cartan pair, the set  $\overline{U'_C}$ is a finite union of disjoint closed basic sets.  If it is a product set, then $c$  has a slice--preserving logarithm $\gamma = \Log_* c$ on ${U}'_C$ by \cite{GPV, AdF2}. If $U'_C \cap \R \ne \varnothing$ and since
the sign of $c|_{\R \cap U}$ is positive, i.e. $c \in \mathcal{SR}^+_{\R}(U_C),$ then   again it has a slice--preserving logarithm $\gamma = \Log_* c$ on ${U}'_C$ by \cite{GPV,AdF2}.  Choose $\delta > 0.$  By Runge Theorem  \ref{RungeV}  we can approximate $\gamma$ on  the compact set $\overline{U_{C}}$ by a
slice--preserving slice--regular function $\gamma'$ on $\H$ as well as we want, in particular, we can have $|\gamma -\gamma'|_{\overline{U_C}} < \delta.$
By Lemma \ref{CSL0} (a), there exist functions $\alpha' \in \mathcal{SR}_{\R}(U_A),$ $\beta' \in \mathcal{SR}_{\R}(U_B)$ such that $\gamma - \gamma' = \alpha'+ \beta'$ and
the estimates  $|\alpha'|_{{A}},|\beta'|_{{B}} < D |\gamma - \gamma'|_{\overline{{U}_C}}< \delta D.$
 Define $\alpha:=\alpha'$ and $\beta:= \beta' + \gamma',$ $a= \exp_*\alpha = \exp \circ \alpha,$ $b:= \exp_* \beta = \exp \circ \beta.$
Then $c = ab.$
The estimate  $|\alpha|_{{A}} < \delta D $ implies that for a small enough $\delta$ we have $|a - 1|_{U_A} < \ve.$
It is clear that also $a \in \mathcal{SR}_{\R}^+(U_A),$ $b \in \mathcal{SR}_{\R}^+(U_B).$

If $\overline{U_C}$ is a basic set then  we may assume without loss of generality that also $U'_{C}$ is a basic set.
If $U_C'$ is a product se, then we proceed as before.
Assume that $\R \cap {U_C'} \ne \varnothing$. Then  $c|_{\R \cap {U_C'}}$ has constant sign.
If it is negative, then $\gamma:= \Log_*(-c)$ exists and $-c = a b$ for $a \in \mathcal{SR}_{\R}^+(U_A),$ $b \in \mathcal{SR}_{\R}^+(U_B)$, hence we can write $c = a (-b).$

{\bf Proof for  $\mathcal{SR}_{\omega}$.}
Let $\omega$ be a global vectorial class and $g_v$ its minimal global representative. The conditions that $U'_C$ is a union of closed basic domains and  
$g_v \ne 0$ on $U$ ensure by \cite{GPV,AdF2}, that $\gamma:= \Log_*c \in \mathcal{SR}_{\omega}(U_C')$
exists and, moreover, the set  $\overline{U_C}$ is Runge in $\H$ (and in $\Omega$). 
Decompose $\gamma= \gamma_0 + \gamma_1 g_v.$ 
We first approximate $\gamma_l$ on  $\overline{U_C}$ by $\gamma_l'\in \mathcal{SR}_{\R}(\bH)$, $l = 0,1$, and hence $\gamma$ by $\gamma' \in \mathcal{SR}_{\omega}(\bH)$ and then conclude as in the  proof for $\mathcal{SR}_{\R}$, because for functions in the same vectorial class the equality
$\exp_*(\alpha + \beta) =(\exp_*\alpha) * (\exp_*\beta)$ holds.

{\bf Proof for $\mathcal{SR}$.}  
Because in general $\log_*c$ may not exist on $U_C$ and the $*$-multiplication is not commutative in $\mathcal{SR}$, the proof requires different tools and is  formulated as Theorem \ref{MCSLM} below.
\end{proof}


 The additive Cartan's Splitting
lemma for slice--regular functions (Lemma \ref{CSL0}) together with the matrix representation of
any $f\in \mathcal{SR}(\Omega)$ (with $\Omega$ an axially symmetric domain in $\bH$) allows us to prove the multiplicative Cartan's Splitting lemma   as in \cite{GR}; 
the multiplicative Cartan's splitting lemma in \cite{GR}, VI.E, Theorem
7, used in one complex variable and for the matrix functions $M_f$ on
symmetric domains, boils down to the following theorem because we have at our disposal
the
analogue of VI.E Lemma 6 in \cite{GR},  i.e. the additive Cartan's Splitting
lemma \ref{CSL0}.

\begin{theo}\label{MCSLM} Assumptions as in Lemma \ref{MCSL0}. Let  $M_c \in \mathcal{GL}_{\mathcal{SR}}(U).$  Then there exist matrix functions
$M_a \in \mathcal{GL}_{\mathcal{SR}}(U_A),$ $M_b \in \mathcal{GL}_{\mathcal{SR}}(U_B)$ such that
$M_a M_b = M_c.$ In addition,  given $\ve > 0,$ the function $M_a$ can be chosen to satisfy
\[|M_a - \mathbf{1}|_A < \ve.\]
\end{theo}

\begin{remark}\label{notcm}
  The original result in \cite{GR}, VI.E Theorem 7, is stated for the sets $K, K'$ and $K''$ which in one complex dimension are rectangles. In  our case we set  $K':= A_I, K{''}:= B_I, K:=C_I$,  $F:=M_c|_{C_I}$ and the resulting matrix functions are  $F':=M_a|_{A_I}$ and $F''=M_b|_{B_I}$.
  Recall that  the sets $A_I$ and $B_I$ have the separation property.
  Despite the fact that in our setting we do not  have such a nice geometric situation, the proof is adapted to our situation more or less verbatim. In the sequel we focus only to parts of the proof where we make some small changes in order to have the desired estimates.

  The matrix norm $| \cdot|_2$ used in \cite{GR} is induced by the Euclidean norm  and is of course equivalent to the matrix norm $| \cdot|$ we use, which is induced by vector norm $|\cdot|_1$: for a $4 \times 4$ matrix $M$ we have $|M| / 2 \leq |M|_2 \leq 2|M|$.

  The proof starts by defining nested sequences 
  $A_n = A_I + B(0,2^{-n}\delta),$ $B_n = B_I + B(0,2^{-n}\delta),$  for a small enough $\delta$ so that $(\overline{A_n}, \overline{B_n})$ form a Cartan pair; evidently the lengths of the boundaries of $C_n=A_n \cap B_n$ are bounded from above by some positive constant
  $L$, because the set $C_I=K$ has slice--piecewise smooth boundary.

In the sequel we examine the estimates in \cite{GR}  for the matrix functions.
The proof is divided into two steps. First it is
assumed that $|F - \mathbf{1}|_{2,\overline{C_1}} < \rho/4$ for some small enough $\rho$
(fulfilling also some other conditions but this is not relevant for
the conclusions on the estimates of the product, once we know it
converges).  The function $F'$ is
obtained as an infinite product $F'=
\lim_{n \rightarrow \infty}(\mathbf{1} + G_1') \cdot \ldots \cdot(\mathbf{1} + G'_n)$ with the estimates of the form $|G'_n|_{K'}
< 2^{-n} D \rho$ for some positive constant $D,$ so that also $D \rho
< 1/2.$ Then we can estimate
\begin{eqnarray*}
    |F' - \mathbf{1}|_{2,K'} &\leq& e^{\sum_{n=1}^{\infty}|G'_n|_2} - 1 \leq e^{\rho 2 D } - 1
    \leq  \rho \ (2D e).
\end{eqnarray*}
Hence, if $\ve > 0$ is given, one can choose $\rho $ to be so small, that $4D e\rho  < \ve,$ then
we get the estimate $ |M_a - \mathbf{1}|_A \leq 2 |F' - \mathbf{1}|_{2,K'}  < \ve.$ A similar  estimate holds for
$M_b$ on the set $B.$

If $|F-\mathbf{1}|_{2,\overline{C_1}}$ is not small
enough, then by Runge theorem (Theorem \ref{RungeM} rewritten in matrix representation form) we can approximate $M_c^{-1}$ on  $C_1$ by an invertible matrix function $M^{-1}$ defined on $\H$ so that
$|M_c^{-1} - M^{-1}|_{2,\overline{C_1}} < \rho/(4 |M_c|_{2,\overline{C_1}})$.
Then $|M_c M^{-1}  -\mathbf{1} |_{2,\overline{C_1}} < \rho/4.$
As a result we get the decomposition $M_c M^{-1} = M_a M_b $ together with estimates on $M_a - \mathbf{1}, M_b - \mathbf{1}.$ The desired
decomposition is then $M_c = M_a M_b M $ with estimates on
$M_a-\mathbf{1}$ still valid, while the size of $M_bM-\mathbf{1}$ cannot be
controlled. Similarly, if we consider the decomposition
 $M^{-1}M_c  =M_a M_b $
together with estimates on $M_a, M_b$, then
$M_c = M M_a M_b$ with estimates on $M_b-\mathbf{1}$ still valid, while the size of $M M_a -\mathbf{1}$ cannot be controlled.
\end{remark}

\begin{remark}\label{ChangedOrder}
In view of Remark \ref{notcm}, one immediately concludes that in the statement of Lemma \ref{MCSL0}
the order of product  and the roles of sets $A$ and $B$ can be interchanged.
\end{remark}

\subsection{Solutions to the First and Second Cousin  problem}


As already mentioned, to prove Theorem \ref{H^1SRR+} (2) we have to find a global minimal representative for a given vectorial class.
Let us first prove the following lemma which is based on the fact that the Cartan covering has order $2$ when restricted to the real axis.
\begin{lem}\label{PositiveCocycles} Let $\Omega$ be an axially symmetric domain in $\H$. 
  Given a cocycle $v \in H^1(\Omega,\mathcal{SR}^*_{\R})$ there exists a  Cartan covering $\mathcal{U},$ the representative
  $\{v_{kl}\}_{ k,l \in \bN_0}$ of $v$ and constants $c_l \in \{-1,1\}, l \in \bN_0$, so that $c_k^{-1} v_{kl}c_l  \in \mathcal{SR}^{*,+}_{\R}(U_{kl}), k,l \in \bN_0.$
Moreover,
\[H^1(\Omega,\mathcal{SR}^*_{\R}) \simeq
H^1(\Omega,\mathcal{SR}^{*,+}_{\R}).\]
\end{lem}

\begin{proof}
First notice that every cocycle in $\mathcal{SR}^{*,+}_{\R}$ is also a cocycle in $\mathcal{SR}^{*}_{\R}$.
If two cocycles with values in $\mathcal{SR}^{*,+}_{\R}$ are equivalent in $\mathcal{SR}^{*,+}_{\R}$, then they are also equivalent in $\mathcal{SR}^{*}_{\R}$.
So we have a mapping $H^1(\Omega,\mathcal{SR}^{*,+}_{\R}) \rightarrow H^1(\Omega,\mathcal{SR}^{*}_{\R})$.

We claim that we also have a mapping $H^1(\Omega,\mathcal{SR}^{*}_{\R}) \rightarrow H^1(\Omega,\mathcal{SR}^{*,+}_{\R})$ and it is bijective.
To see this, choose a  cocycle representing a given equivalence class with values in $\mathcal{SR}^*_{\R}$. Every such class can be represented by functions $v_{kl} \in \mathcal{SR}^{*}_{\R}(U_{kl})$ for a sufficiently fine  Cartan covering $\mathcal{U} =\{U_l\}_{l \in\bN_0}$.
If $U_{kl} \cap \R =\varnothing$ then $v_{kl} \in \mathcal{SR}^{*,+}(U_{kl}).$  If $U_{kl} \cap \R \ne \varnothing$, then we try to find a different representative of the cocycle.
More precisely, we are looking for constants $c_m \in \{-1,1\}$ so that $c_k^{-1}v_{kl} c_l$ is in $\mathcal{SR}_{\R}^{*,+}$.  We start by defining $c_l = 1$ if $U_l \cap \R = \varnothing.$

 As already noted, the only $v_{kl}$-s which might  not be in $ \mathcal{SR}^{*,+}(U_{kl})$ are the ones where $U_{kl} \cap \R \ne \varnothing$.  By construction of the Cartan covering, each connected component of $\Omega \cap \R$ is covered by a chain of topological open balls which do not intersect any other connected component of $\Omega \cap \R$. We define the suitable constants $c_l$ on each component separately.

 Let $(x_0,x_1) \subset \R$ be a connected component of $\Omega \cap \R$ and let $\mathcal{V} = \{V_l\}_{l \in \Z}$ be the family of all the sets $U_n$ with $U_n \cap  (x_0,x_1) \ne \varnothing.$ These are  topological open balls  which form a chain on $(x_0,x_1),$ because by definition of a Cartan covering, we have
$U_{ijk} \cap  \R = \varnothing$ for distinct $i,j,k.$  Fix one of these sets and denote it by $V_0.$ Then inductively denote by $V_{l+1}$ the set which
intersects $V_l$ on  the right-hand side and by $V_{l-1}$ the set which intersects $V_l$ on the left-hand side.
For $l \in \Z,$ let $c_{l,l+1} = 1$ if the sign of $v_{l,l+1}|_{V_{l,l+1} \cap \R}$ is positive and $c_{l,l+1} = -1$ otherwise.
Define $c_0 = 1$ and inductively define $c_{l+1} = \Pi_{k=0}^lc_{k,k+1}$ for $l \geq 0.$  Similarly, for $l \leq 0,$ set $c_{l} = \Pi_{k = l}^0 c_{k,k+1}.$ Then $c_{l,l+1} = c_l^{-1} c_{l+1}.$
Replace the representatives $v_{l,l+1}$ on the sets $V_{l, l+1}$ by $c_l^{-1}v_{l,l+1}c_{l+1},$ $l \in \bN_0.$
The new representatives define the same
cocycle and have values in $\mathcal{SR}_{\R}^{*,+}.$

Now assume that  a given equivalence class $[v]$ is represented by the cocycle $\{v_{kl}\}_{k,l\in\bN_0}$ and also has  representatives $\{u_{kl}\}_{k,l\in\bN_0}$ and $\{w_{kl}\}_{k,l\in\bN_0}$
with values in $\mathcal{SR}_{\R}^{*,+}$. Then $v_{kl}=c_ku_{kl}c_l^{-1} = d_kw_{kl}d_l^{-1}$ for some $d_k \in \mathcal{SR}^*_{}(U_k), k \in \bN_0$. Hence $u_{kl} = (c_k^{-1}d_k) w_{kl} (c_l^{-1}d_l)^{-1}.$ Because $w_{kl}$ and $u_{kl}$ are in the class $\mathcal{SR}^{*,+}$ and the product commutes, also $(c_k^{-1}d_k)(c_l^{-1}d_l)^{-1}$ is in $\mathcal{SR}^{*,+}(U_{kl})$. If $U_{kl}$ intersects the real axis, it is a basic domain and so either both $e_k:=c_k^{-1}d_k$ and $e_l:=c_l^{-1}d_l$ are in the class $\mathcal{SR}^{*,+}$ or $-e_k$ and $-e_l$ are in the class $\mathcal{SR}^{*,+}$. Consider the covering $\mathcal{V}$ containing the real interval $(x_0,x_1)$ as above and assume that $e_0 \in \mathcal{SR}^{*,+}_{\R}(V_0).$ Then also $e_{-1}$ and $e_1$ are in the class $\mathcal{SR}^{*,+}_{\R}$ and consequently this holds for all sets that intersect $(x_0,x_1)$. If $-e_0\in \mathcal{SR}^{*,+}_{\R}(V_0)$ then replace $e_k$ by $-e_k$ and again $u_{kl}$ and $w_{kl}$ are conjugated by functions from the class $\mathcal{SR}^{*,+}_{\R}$.

So the mapping that maps the equivalence class of the cocycle $\{v_{kl}\}_{k,l\in\bN_0}$ with values in $\mathcal{SR}_{\R}^{*}$ to the equivalence class of the cocycle $\{u_{kl}\}_{k,l\in\bN_0}$ with values in $\mathcal{SR}_{\R}^{*,+}$ is well-defined.

Moreover, we have proved that if two cocycles  are equivalent in $\mathcal{SR}_{\R}^{*},$ they are equivalent in $\mathcal{SR}_{\R}^{*,+}.$ Indeed, take two cocycles with values in $\mathcal{SR}^{*}_{\R}$ and choose representatives of their equivalence classes to be $\{u_{kl}\}_{k,l \in \bN_0}$ and $\{v_{kl}\}_{k,l \in \bN_0}$ in $\mathcal{SR}^{*,+}_{\R}$  with respect to some Cartan covering $\{U_l\}_{l \in \bN_0}$. Then  there exist $c_l \in  \mathcal{SR}^{*}_{\R}(U_l)$ so that $v_{kl} = c_k^{-1} u_{kl} c_l$. Using the above notation consider  $\mathcal{V}$, the covering of the real interval $(x_0,x_1)$  and assume without loss of generality that  $c_0 \in \mathcal{SR}^{*,+}_{\R}(V_0)$ (otherwise multiply all $c_l$ by $-1$).  But then also $c_l = u_{l0}c_0v_{0l}\in \mathcal{SR}_{\R}^{*,+}(V_{0l})$ for $l = 1, -1$ and inductively, since the sets $V_l$ are basic, we have $c_l  \in \mathcal{SR}_{\R}^{*,+}(V_{l})$.
This means that the mapping $H^1(\Omega,\mathcal{SR}^{*}_{\R}) \rightarrow H^1(\Omega,\mathcal{SR}^{*,+}_{\R})$ is injective and surjective, i.e.
$H^1(\Omega,\mathcal{SR}^*_{\R}) \simeq
H^1(\Omega,\mathcal{SR}^{*,+}_{\R}).$
\end{proof}

\begin{lem}\label{vc} The vectorial class $\omega$ on an axially symmetric domain  $\Omega$ defines a multiplicative cocycle in $H^1(\Omega,\mathcal{SR}_{\R}^{*,+}).$ More precisely, for every sufficiently fine Cartan covering $\mathcal{U}=\{U_l\}_{l\in \bN}$ there is a family of minimal local representatives $\{v_l\}_{ l \in \bN}$ with $v_{kl}:=v_k * v_{l}^{-*} \in \mathcal{SR}^{*,+}(U_{kl}).$
\end{lem}

\begin{proof}
Let $\mathcal{U} = \{U_l\}_{l\ \in\bN}$ be a Cartan covering and for each $l \in \bN$  let $v_l$ be a minimal representative of $[v].$
By definition of  vectorial class, we have $v_{kl}:= v_{k}*v_l^{-*} \in \mathcal{SR}_{\R}^*(U_{kl}),$ so it defines a (multiplicative) cocycle in $\mathcal{SR}_{\R}^*(\Omega).$ By Lemma \ref{PositiveCocycles} there exist a representative of the same cocycle with functions in $\mathcal{SR}^{*,+}_{\R}.$
\end{proof}
Now we can prove Theorem \ref{H^1SRR} (1). For its solution
we will make use of the function $\log_*$; as expected this approach
requires very accurate description of the cocycles in the different
settings considered.
\begin{proof}[Proof of Theorem \ref{H^1SRR} (1)]
Fix an element in $H^1(\Omega, \mathcal{SR}_{\R}^{*,+}).$ Then there exists a Cartan covering $\mathcal{U} = \{U_l\}_{l\in\bN_0}$ of the axially symmetric domain $\Omega$ and let $\{(v_{kl}, U_{kl}), v_{kl}\in \mathcal{SR}_{\R}^{*,+}(U_{kl})\}_{ k,l \in \bN}$
 be the given multiplicative cocycle in $\mathcal{SR}_{\R}^{*,+},$
 namely $v_{kl} \in \mathcal{SR}_{\R}^{*,+}(U_{kl})$, $l,k \in \bN_0$, are such that they satisfy conditions  $v_{kl}* v_{lk} = 1$ on $U_{kl}$ and
  $v_{kl}* v_{lm} * v_{mk} =1$ on $U_{klm}$.

 Thanks to \cite{GPV}
there exist  slice--regular logarithms of $v_{kl}$, call them $\Log_* v_{kl} \in \mathcal{SR}_{\R}(U_{kl})$; the cocycle conditions
    together with the fact that $U_{klm}$ is basic, imply that
\begin{equation}\label{cc}
  \Log_* v_{kl} + \Log_* v_{lm} + \Log_* v_{mk} =  2\pi {\mathcal I} n_{klm}.
\end{equation}
This induces the cocycle $C_{co}=\{({\mathcal I} n_{klm}, U_{klm})\}_{ k,l,m \in \bN}$ which is trivial by  Theorem \ref{HNIZ} 
 i.e. there exists a cochain $C=\{(\mathcal I n_{kl}, U_{kl})\}_{k,l \in \bN}$ of slice--preserving functions which is mapped to $C_{co}$ by the coboundary operator.

Therefore functions $\Log_* v_{kl} - 2\pi\mathcal{I}n_{kl}$ form a cocycle of slice--preserving functions. Since $H^1(\Omega,\mathcal{SR}_{\R}) = 0$ by Theorem \ref{H^1SRR+},
there exist slice--preserving functions ${v}_l', l \in \bN,$ satisfying $v_{kl}' =v_l' - v_k'= \Log_* v_{kl} - 2\pi\mathcal{I}n_{kl},$ $k,l \in \bN.$
Then $v_l:=\exp_*(-v_l')$ satisfy $v_{kl} = v_k v_l^{-1}.$
\end{proof}

Lemma \ref{vc} and
Theorem \ref{H^1SRR} (1) immediately imply the following.
\begin{theorem}\label{GMR} Each globally defined vectorial class $\omega$ on an axially symmetric domain $\Omega$ is principal, i.e. there exists $ g_v \in \mathcal{SR}_{\omega}(\Omega),$ so that $\omega = [g_v]$.  Moreover, $Z(\omega) = g_v^{-1}(0),$ with multiplicities, so $g$ is a global  minimal representative of $\omega.$
\end{theorem}
\begin{proof} As in the proof of Lemma \ref{vc}, let $\mathcal{U} = \{U_l\}_{l\ \in\bN}$ be a Cartan covering of $\Omega$, for each $l \in \bN$  let $v_l$ be the minimal representative of $[v]$ in $U_l$ and  let,  without loss of generality,  $v_{kl}:= v_{k}^{-*}*v_l \in \mathcal{SR}_{\R}^{*,+}(U_{kl}),$ be the induced cocycle.

Because $H^1(\Omega, \mathcal{SR}^{*,+}_{\R}) = 0,$ $v_{kl} =  g_k^{-1}* g_l$ with $g_l,g_k$ slice--preserving, so
$v_l*g_l^{-*}  = v_k* g_k^{-*}$ so $\{(v_l*g_l^{-*}, U_l), l \in \bN\}$ defines a function on $\Omega.$
\end{proof}
\begin{proof}[Proof of Theorem \ref{H^1SRR+} (2)] Granted the existence of a global minimal representative of $\omega$, the proof is the same as the proof of claim (1) in Theorem \ref{H^1SRR+}, with the difference that the Cartan coverings have to be subordinated to $Z[\omega]$, and we have to use Lemma \ref{CSL0} (b) in the place of Lemma \ref{CSL0} (a).
\end{proof}
Now we can also prove  Runge Theorem \ref{RungeV} (2).

\begin{proof}[Proof of Theorem \ref{RungeV} (2)]
Let $g_v$ be a global minimal representative and $K \Subset U.$ Let ${K}':= \widetilde{K}$ be the symmetrization of $K$ and
 $\gamma  = \gamma_0 + \gamma_1 g_v \in \mathcal{SR}_{\omega}(U)$ the decomposition of $\gamma$.  Since  Runge theorem holds for slice--preserving slice--regular functions, $\gamma_l, l = 0,1$ can be
approximated on ${K}'$  by slice--preserving functions $\tilde{\gamma}_l'$  and hence
${\gamma}':= {\gamma}_0' + {\gamma}_1' g_v$
approximates $\gamma$ on $K.$
\end{proof}


\begin{remark}
For a class $\mathcal{F}$ of functions on a domain $\Omega$ which is closed under multiplication,
the second (or multiplicative) Cousin problem can be stated in this version: given an open covering $\mathcal{U}=\{U_{\alpha}\}_{\alpha}$ of a domain $\Omega$ and consider functions $f_{\alpha \beta} \in \mathcal{F}(U_{\alpha\beta})$  which satisfy the {\em (multiplicative) cocycle} or {\em compatibility} conditions
\begin{equation}\label{xcomp}
f_{\alpha\beta}\cdot f_{\beta\alpha}=1\,\, {\mathrm{in}}\ \ U_{\alpha\beta}\quad \mbox{ and }\quad
f_{\alpha\beta}\cdot f_{\beta\gamma}\cdot f_{\gamma\alpha}=1\,\, {\mathrm{in}}\ \ U_{\alpha\beta\gamma}
\end{equation}
for all $\alpha,\beta,\gamma$; it is required to find functions $h_{\alpha} \in \mathcal{E}(U_{\alpha})$ and in the same class, such that $f_{\alpha\beta}=h_{\alpha}^{-1}\cdot h_{\beta}$
for all $\alpha,\beta$.

In the noncommutative case, one would also look for the solution of the form $f_{\alpha\beta}= h_{\beta}\cdot h_{\alpha}^{-1}$
for all $\alpha,\beta$.
In this case, the  cocycle condition  is of the form $f_{\gamma\alpha}\cdot f_{\beta\gamma}\cdot f_{\alpha\beta}=1$.
 The order of writing the cocycle condition does not affect the solvability of the second Cousin problem, since they are linked by inverting the data, $ f_{\alpha \beta} \ra f_{\beta \alpha} = f_{\alpha\beta}^{-1}.$ It rather changes the side where factors $h_{\alpha}$ appear.
If the cocycle is given by functions $f_{\alpha},$ $f_{\alpha\beta}=  f_{\alpha}^{-1} \cdot f_{\beta}$, then if
 $f_{\alpha\beta}=  f_{\alpha}^{-1}\cdot f_{\beta} =h_{\alpha}^{-1}\cdot h_{\beta},$ the functions  $\{f_{\alpha}\cdot h_{\alpha}^{-1}\}$ define a global function. Similarly, if we have
$f_{\alpha\beta}= f_{\beta}\cdot f_{\alpha}^{-1}= h_{\beta}\cdot h_{\alpha}^{-1}$, then $\{h_{\alpha}^{-1}\cdot f_{\alpha}\}$ define a global function on $\Omega$.

Hence, if $\mathcal{F}$ is one of the classes $\mathcal{SR}_{\R},\mathcal{SR}_{\omega}$ or $\mathcal{SR}$ and we want to preserve properties of
the initial data such as zeroes and poles of semiregular  or regular functions, which are preserved under $*$-multiplication from the right, then the first version has to be applied.
\end{remark}

\begin{proof}[Proof of Theorem \ref{H^1SRR} (2),\ (3)] Let us first prove $(3)$.
 Let an element in $H^1(\Omega,\mathcal{SR}^*)$ be represented by a
  cocycle defined on a sufficiently
  fine Cartan covering $\mathcal U = \{U'_k\}_{ k \in \bN_0}. $
  Let $\mathcal{B} = \{U_k\}_{ k \in \bN_0}$ and  $\mathcal{B}^{1/(j+1)} = \{U_k^{j}\}_{k \in \bN_0} $ be Cartan coverings given by { Theorem 1.1 in \cite{PV2}. } 
  Let the cocycle be represented by the functions $v_{kl} \in \mathcal{SR}^*(U^0_{kl}),$ i.e the functions satisfy the conditions
\[v_{kl} * v_{lk} = 1,\quad v_{klm}:= v_{mk}*v_{lm}*v_{kl} = 1.\]

 To find nonvanishing functions $v_k$ defined on $U_k,$ $k \in \bN_0$ so that $v_{kl} = v_l * v_k^{-*}$ on $U_{kl}$, we follow the scheme of the proof is the same as the one for Theorem \ref{H^1SRR+} (1), with the difference that  we have to use multiplicative Cartan's splitting lemma \ref{MCSL0} with the decompositions of the form $c = b_n*a_n$  and estimates of the form $|a_{n+1} - 1|_{A} < \ve/2^{n+1}$
in the place of additive Cartan's splitting lemma \ref{CSL0} and estimates $|a_{n+1} - 0|_A < \ve/2^{n+1}$.
The estimates then ensure that the product $\ldots a_{k+3} * a_{k+2} * a_{k+1}$ converges uniformly on $\overline{U_k}$.

 For the proof of (2) and vectorial class $\omega$, let the coycle be represented by functions $v_{kl} \in \mathcal{SR}^{*}_{\omega}(U_{kl})$ for a  sufficiently fine Cartan covering $\mathcal{U}$  subordinated to the pair $(\{\Omega\},Z(\omega))$. By assumption, $U_{kl} \cap Z(\omega) = \varnothing$ for each $k \ne l$ and hence the Cartan's Splitting lemma \ref{MCSL0}  (b) for vectorial class applies.  Then the proof goes as the case of slice--regular functions with the only difference that the splitting is done within a given vectorial class.
\end{proof}
\begin{remark}\label{coc}
It is not possible to reduce
the multiplicative Cousin problem in a given vectorial class to the problem in $H^2(\Omega, \mathcal I \Z).$ Let $g_v$ be the
minimal representative of $\omega$ on $\Omega.$ 
Every cocycle can be represented by functions $v_{kl} \in
\mathcal{SR}^{*}_{\omega}(U_{kl})$ for some sufficiently fine  Cartan covering $\mathcal{U}$ subordinated to $(\{\Omega\},  Z(\omega))$, so
 $U_{kl} \cap Z(\omega) = \varnothing.$
{In  \cite{GPV} (see also Proposition \ref{TTT}) it is shown that}
there exist logarithms $\Log_* v_{kl} \in \mathcal{SR}_{\omega}(U_{kl})$. The cocycle condition then implies that
$$
  \Log_* v_{kl} + \Log_* v_{lm} + \Log_* v_{mk} =  2\pi {\mathcal I} n_{klm} + 2 \pi \nu_{klm}\mathcal{J}, \; \mathcal{J}:=\frac{g_v}{\sqrt{g_v^s}}.
  $$ 
 By Theorem \ref{HNIZ} we know that the cocycle $\{{\mathcal
    I}n_{klm}, k,l,m \in \bN_0\}$ is trivial. Unfortunately, the
  imaginary unit function $\mathcal{J}$
 is only locally defined and hence this data do not induce (at least not in an obvious way) a cocycle problem with integer coefficients. The problem is not that $ g_v$ may vanish since we have  chosen the covering in such a way that $g_v$ does not have zeroes on the sets $U_{kl}.$ The problem is that the square root $\sqrt{g_v^s}$ is not necessarily globally defined since the domain which is the union of $U_{kl}-s$ is not basic.
 \end{remark}

Next proposition is now straightforward:
\begin{proposition} Let $\Omega$ be an axially \ssd and $\omega = [g_v]$ a vectorial  class with vectorial function $g_v$ as its global minimal representative. Assume that also $\sqrt{g_v^s}$ is globally defined.

If $Z(\omega) = \varnothing$, then $\Gamma(\Omega,{\mathcal K}_{\omega})  = \mathbb{Z}$.

If $Z(\omega) \ne \varnothing,$ then
$
  \Gamma(\Omega,{\mathcal K}_{\omega}) = 0.
$
\end{proposition}

\begin{proof}
Notice that, because the domain $\Omega$ intersects the real axis, the periods with $\mathcal{I}$ are excluded.
If the vectorial class does not have a zero, then a section from  $\Gamma(\Omega,{\mathcal K}_{\omega})$ is locally equal to $ 2 n \pi \frac{g_v}{\sqrt{g_v^s}}$ for some $n \in \Z$ and the identity principle implies that this holds globally.

If the vectorial class has a zero  $z_0 \in Z(\omega),$ then in a neighbourhood of $z_0$, $\log_* 1=2 \pi n \mathcal{I}$  and hence, by identity principle it is constant on every slice. Since the slices intersect the real axis, $n =0.$
\end{proof}

 The following is a generalization of Theorem \ref{H^1SRR+}. Notice that the existence of a globally defined minimal representative of $\omega$ is granted thanks to Theorem \ref{GMR}.
\begin{theorem} Le $\Omega$ be an axially symmetric domain and let
$\mathcal{F} \in \{\mathcal{SR}_{\R},\mathcal{SR}_{\omega},\mathcal{SR}\}$. Then
$ H^q(\Omega,\mathcal{F})=0$
 every $q> 0$.
\end{theorem}
\begin{proof} The triviality for $q=1$ was proved in Theorem \ref{H^1SRR+}. Triviality of higher cohomology groups follows directly from the fact that  a cocycle in $\mathcal{F}$ can be represented on a sufficiently fine Cartan covering of $\Omega$ and we have the corresponding additive splitting lemmas; in the case of $\mathcal{F}=\mathcal{SR}_{\omega}$ the covering has to be subordinated to $(\{\Omega\},Z[\omega])$. Any Cartan covering  has order $3$, hence the groups $H^q(\Omega,\mathcal{F})$ are trivial for $q \geq 3$. 

Assume that $q = 2$ and let an element in $H^2(\Omega,\mathcal{F})$ be represented by a
  cocycle defined on a sufficiently fine Cartan covering $\mathcal U = \{U'_k\}_{ k \in \bN_0} $ of $\Omega$.
  Let $\mathcal{B} = \{U_k\}_{ k \in \bN_0}$ and  $\mathcal{B}^{1/(j+1)} = \{U_k^{j}\}_{k \in \bN_0} $ be the Cartan coverings given by { Theorem 1.1 in \cite{PV2}. }
Let the cocycle on the Cartan covering $\mathcal {B}^1 = \{U^0_j\}_{ j \in \bN_0}$ be given by functions $v^0_{klm} \in \mathcal{F}(U^0_{klm}), k,l,m \in \bN_0$.
The claim follows from the fact  that the decomposition  of the form $v_{klm}= v_{kl} + v_{lm} + v_{mk}$ can be done inductively, since there is no cocycle condition to check. 
We briefly sketch the idea of the proof.

$n = 1.$ Set $v^0_{01} = 0$  and consider $v^0_{012}\in \mathcal{F}(U^0_{012})$ if  $U^0_{012} \ne\varnothing$.
Decompose $v^0_{012} -  v^0_{01}$ defined on
$U^0_{12} \cap U^0_{02}$ using the Cartan's splitting lemma \ref{CSL0} to get $ v^0_{012} - v^0_{01}= v^1_{12} + v^1_{20}$ with $v^1_{12}$ defined on $U^1_{12}$ and $v^1_{02} = -v_{20}^1$ defined on $U^1_{02}$. Set $v_{01}^1:=v_{01}^0|_{U_{01}^1}$. If $U^0_{012} =\varnothing$, set
$v_{02}^1 = v_{12}^1 = 0.$

$n-1 \ra n.$ Assume that $v^{n-1}_{kl} \in \mathcal{F}(U_{kl}^{n-1})$, $l,k = 0,\ldots n-1$, have already been constructed.
By definition, the sequence $\{\overline{U^{n-1}_{0 n}},\ldots,\overline{U^{n-1}_{n-1\, n}}\}$ is a Cartan string and hence the covering
$\{U^{n-1}_{0n},\ldots,U^{n-1}_{n-1\, n}\}$ of $U^n:=U^{n-1}_{0n}\cup\ldots\cup U^{n-1}_{n-1\, n}$ has order $2$ so no three sets from the Cartan string intersect.
In particular, the nerve of this covering  consists of a finite number of disjoint 'chains', i.e. simplicial complexes which are   finite sequences of points and line segments.

Consider the set $U_{kln}^{n-1}$ for $k<l<n$ and assume that it is not empty. We would like to show that the sets $\overline{U^{n-1}_{k n}}$ and $\overline{U^{n-1}_{l n}}$ form  a Cartan pair. Because the sequence
$\{\overline{U^{n-1}_{0n}},\ldots,\overline{U^{n-1}_{l n}}\}$ is a Cartan string, the separation property is fulfilled for the sets
$A = \overline{U^{n-1}_{0n}}\cup \ldots \cup \overline{U^{n-1}_{l-1\, n}}$ and $B = \overline{U^{n-1}_{l n}}$. Because  $A \cap B = \overline{U_{kln}^{n-1}},$
the separation property is fulfilled for the sets $\overline{U^{n-1}_{k n}}$ and $\overline{U^{n-1}_{l n}}$ thus making them a Cartan pair.

The problem can be solved on each of the chains separately. Take one of the  above chains. If it consists only of a point, it corresponds to a set $U_{kn}^{n-1}$ for some $k$ and we set $v_{kn}^n:=0$. Else, assume, for simplicity, that
the chain is defined by the sets $U^{n-1}_{0 n},\ldots,U^{n-1}_{k n}$ and proceed by induction.
Choose a  pair of sets $(U^1_A, U^1_B)$ so that $U^{n-1}_{0n}\Supset U^1_A \Supset U^n_{0n}$,   $U^{n-1}_{1n}\Supset U^1_B \Supset U^n_{1n}$ and $(\overline{U^1_A}, \overline{U^1_B})$ is a Cartan pair. Then
$v^{n-1}_{01n}$ decomposes as $v^{n-1}_{01n} - v^{n-1}_{01} = v_{1n}- v_{0n}$ with $v_{0n} \in \mathcal{F}(U^1_A)$ and  $v_{1n} \in \mathcal{F}(U^1_B)$.
Assume that $v_{ln}$ have already been constructed for $l \leq m < k$ and let $(U^m_A , U^m_B)$ be the last pair of sets. Choose a  pair of sets $(U^{m+1}_A, U^{m+1}_B)$ so that  $U^m_A \cup U^m_B \Supset U^{m+1}_A \Supset (U^{n}_{0n} \cup \ldots \cup U^{n}_{mn})$,   $U^{n-1}_{m+1\,n}\Supset U^{m+1}_B \Supset U^n_{m+1\,n}$ and $(\overline{U^{m+1}_A}, \overline{U^{m+1}_B})$ is a Cartan pair. Then
$v^n_{m\, m+1\, n}$ decomposes as $v^n_{m\, m+1\,n} - v^{n-1}_{m\,m+1} - v_{nm} = v_{m+1\,n}- a$ with $a \in \mathcal{F}(U^{m+1}_A)$ and  $v_{m+1\, n} \in \mathcal{F}(U^{m+1}_B)$. Put $v_{ln}':=v_{ln} - a$ for $l = 0,\ldots,m$. Then
$v^n_{l_1l_2n} = v^n_{l_1l_2} + v_{l_2n}' + v_{nl_1}' = v^n_{l_1l_2} + v_{l_2n} -a + v_{nl_1}+a = v'_{l_1l_2n}$. Replace $v_{ln}$ by $v_{ln}'$.
When $m = k$ set $v^n_{ln}:= v_{ln}|_{U_{ln}^n}$ for $l \leq k.$ This completes the induction step and the proof.

\end{proof}

\section{Applications to divisors and jet interpolation theorems}\label{applications}

In this section we present some consequences of the previously stated results.

\begin{theorem}[Principal divisors]\label{divisor} Every divisor on an axially symmetric domain $\Omega \subset \H$ is principal.

More precisely, let $\mathcal B = \{U_l\}_{l \in \bN}$ and $\mathcal B^1 = \{V_l\}_{l \in \bN}$ be  Cartan coverings of $\Omega$ given by Theorem  1.1 in \cite{PV2} and let the pairs $\{(V_{l}, f_{l}), f_{l} \in \mathcal{SMR}(V_l)\}$ represent a given divisor. Then there exists a global semiregular function $f \in \mathcal{SMR}(\Omega)$ so that $f_l^{-*} * f \in \mathcal{SR}^*(U_l).$
\end{theorem}
\begin{proof}
 Let $Z$ be the support of the divisor. Without loss of generality we assume that $\{U_j\}_{ j \in \bN}$ is a Cartan covering
 of $\Omega$
 satisfying the condition $U_{kl} \cap Z = \varnothing$ for $k \ne n.$ 
 By Theorem \ref{H1*}, there exist $h_k \in \mathcal{SR}^*$ so that $f_{kl} = h_k*h_l^{-*} = f_k^{-*}*f_l,$ hence $f_l*h_l =f_k*h_k$ so $\{f_l*h_l \in \mathcal{SMR}(U_l), l \in \bN\}$ defines a semiregular function with prescribed zeroes and poles.
\end{proof}

To prove Main Theorem 2 (Theorem \ref{jets1}) we first need
\begin{theorem}[Functions with prescribed zeroes]\label{zeromult} Let $\Omega$ be an axially symmetric domain and $Z$
  a discrete set of points and spheres  contained in $\Omega$.
  As in the Interpolation Theorem \cite{PV1}, write
  $Z = \mathcal S \cup \mathcal{X} \cup \mathcal Z,$ where
\begin{enumerate}
\item $\mathcal S = \{S_j = \{a_j + Ib_j, I \in \mathbb S\}, a_j, r_j \in \R, r_j > 0, I \in \bN\}$ is the set of spheres
\item $\mathcal X = \{x_l \in \R, l \in \bN\}$ is the set of real points,
    \item $\mathcal Z = \{z_k \in \H \setminus \R, k \in \bN\}$ is the set of nonreal points, so that $\widetilde{\{z_k\}} \cap \mathcal{Z} = \{z_k\}$ for each $k \in \bN.$
\end{enumerate}
  Assume that to each sphere $S_j \in \mathcal S$ { an integer} $\sigma_j,$  to each $x_j\in \mathcal X$ an integer $\xi_l$ and to each $z_k \in  \mathcal Z$ an integer $\zeta_j$ is assigned.
  Then there exists $f \in \mathcal {SR}(\Omega)$
 with prescribed multiplicites at $Z;$
the multiplicity of the sphere $S_j$ is $\sigma_j,$ of the point $x_l$ is $\xi_l$ and of the point $z_j$ is $\zeta_j$.
\end{theorem}

\begin{proof}
 By definition, each sphere $S_j$ contains at most one point from the set $\mathcal Z.$ If it exists,  we denote it by $z_j'.$
 Let $\mathcal{P}=\{P_j, j \in \bN\}$ be the set of polynomials so that for each $j \in \bN$
 \begin{enumerate}
  \item $P_{3j}$ has prescribed spherical multiplicities at the spheres $S_j$ and at a point $z_j'$ if it exists, and no other zeroes
  \item $P_{3j-2}$ has prescribed zeroes with multiplicities at the point $x_j$ and no other zeroes and
  \item $P_{3j-1}$ has prescribed zeroes with multiplicities at the point $z_j \not \in (\cup_j S_j) \cap \mathcal Z$ and no other zeroes.
  \end{enumerate}
 Define the spheres
 $\tilde{S}_j:= \widetilde{\{z_j\}}.$ Set $\Omega' = \Omega \setminus \widetilde{Z}.$ Then the sets $U'_{3j}:=\Omega' \cup S_j,$ $U'_{3j-2}:=\Omega' \cup
 \{ x_j\}$
 and $U'_{3j-1}:=\Omega' \cup \tilde{S}_j$ define an axially
 symmetric open covering $\mathcal U$ 
 and $P_k^{-*}*P_l$ define a cocycle.
 As customary, there exists  finer
Cartan coverings $\mathcal B^{1/{(j+1})} = \{U_k^j, k\in \bN_0\}$, $j=0,1$, subordinated to
 $(\mathcal U,Z)$.

 By construction, for each open set $U_j^0
 \in \mathcal B^1$ there is at most one polynomial $P\in\mathcal{P}$ with zero in $U^0_j.$
If there is such a $P$, define $v_j:= P^{-*}$ else define $v_j = 1$ on $U^0_j.$

The functions $v_j$ define a cocycle $u_{kl}:= v_l * v_k^{-*} \in
U_{kl}^0$ of invertible functions.

 The cocycle  is trivial by Theorem \ref{H1*}
  and hence there exist invertible functions $h_k \in \mathcal{SR}^*(U_k^1)$ so that $
 v_k * v_l^{-*}= h_k * h_l^{-*}.$ Since *-multiplication from the
 right hand side preserves the zeroes of the function on the left hand
 side, by setting $u_k:= v_k^{-*}*h_k$ we define a slice--regular
 function on $U_k^1$ with prescribed zeroes on $U_k^1$ and for each $k,l
 \in \bN_0$ the functions $u_k$ and $u_l$ coincide on $U_{kl}^1$, hence
 the family $\{u_k \in \mathcal{SR}(U_k^1), k \in \bN_0\}$ defines a
 global slice--regular function on $\Omega.$
\end{proof}
 We are now ready to proceed to the proof of  Theorem \ref{jets1}.

\begin{proof}[Proof of Theorem \ref{jets1}]
We follow a similar approach as in the proof of Theorem \ref{zeromult}. Let $\Omega$ and $Z$ be as in Theorem \ref{zeromult}.
By definition, each sphere $S_j$ contains at most one point  $z_j'$  from the set $\mathcal Z.$
 Let
 $\mathcal{P}=\{P_j, j \in \bN\}$
 be the set of polynomials so that for each  $j\in
 \bN,$  the polynomial $P_{3j}$ has prescribed spherical jet for the point 
 $z_j'$ if $S_j \cap \mathcal Z = \{z_j'\}$ and a prescribed spherical jet on the sphere $S_j$ otherwise;  similarly, let polynomials
 $P_{3j-2},j \in \bN$ have prescribed jets at the points $x_j$ and $P_{3j-1},j \in \bN$ at the points $z_j \ne z_k'$.

 Let $g$ be a function with prescribed zeroes in $Z$ with multiplicities strictly grater than the  corresponding  orders of jets as in Theorem \ref{zeromult}.

 Define the spheres
 $\tilde{S}_j:= \widetilde{\{z_j\}}.$ Then the sets $U_{3j}:=\Omega \setminus S_j,$ $U_{3j-2}:=\Omega \setminus\{ \tilde{x}_j\}$
 and $U_{3j-1}:=\Omega \setminus \tilde{S}_j$ define an axially
 symmetric open covering $\mathcal U$.
  As before, there exists  finer
Cartan coverings $\mathcal B^{1/{j+1}} = \{U_k^j, k\in \bN\}$, $j \in \bN_0$, subordinated to
 $(\mathcal U,Z)$

 By construction, for each open set $U_k^0
 \in \mathcal B^1$ there is at most one polynomial $P$ with prescribed jet in $U^0_j.$
 If it exists, define $Q_j = P$ else define $Q_j = 0$.

 The functions $v_{kl}:= g^{-*}*(Q_k - Q_l) \in \mathcal{SR}(U_{kl}^0)$ define an additive cocycle  of slice--regular functions which is trivial by Theorem \ref{H^1SRR+} 
  and hence there exists slice--regular functions $v_k\in\mathcal{SR}(U_{k}^1)$   so that  $v_{kl}= v_k - v_l$. Then
 $
   Q_k - Q_l = g*v_k - g*v_l \mbox{ on } U_{kl}^1.
 $ Since *-multiplication from the
 right hand side preserves the zeroes of the function on the left hand
 side, the set $\{Q_k - g*v_k, k \in \bN\}$  defines a global function with prescribed jets.
 \end{proof}


\begin{thebibliography}{99}
\bibitem[AdF1]{AdF1} {\sc A. Altavilla, C. de Fabritiis}, {\em *-exponential of slice regular functions}, Proc. Amer. Math. Soc. 147, 1173-1188, 2019.
\bibitem[AdF2]{AdF2} {\sc A. Altavilla, C. de Fabritiis}, {\em *-logarithm for slice regular functions}, Atti Accad. Naz. Lincei Cl. Sci. Fis. Mat. Natur. 34, no. 2,  491–-529 (2023).
\bibitem[AdF3]{AdF3} {\sc A. Altavilla, C. de Fabritiis}, {\em Equivalence of slice semi--regular functions via Sylvester operators},  Linear Algebra and its Applications 607, 151-189, 2020
\bibitem[CSS]{css}  F. Colombo, I. Sabadini,  D. C. Struppa, {\em Sheaves of slice regular functions}, Math. Nachr. 285, No. 8–9, 949–-958 (2012) / DOI 10.1002/
\bibitem[CGS]{CGS}{\sc F. Colombo, J. O. Gonzalez-Cervantes, I. Sabadini}, {\em The C-property for slice regular functions and applications to the Bergman space}, Compl. Var. Ell. Eq., 58, n. 10, 1355--1372, 2013.
\bibitem[FP]{fp} {\sc F. Forstneri\v c, J. Prezelj}, {\em Oka's principle for holomorphic submersions with sprays}, Math. Ann.,  322 (4),  633--666, 2002.
\bibitem[GPV]{GPV} {\sc G.Gentili, J.Prezelj, F. Vlacci}, {\em On a definition of logarithm of quaternionic functions}, J. Noncommut. Geom. 17 (3), 1099–-1128 (2023).
\bibitem[GSS]{GSS} { G. Gentili, C. Stoppato, D. Struppa}, {\em Regular functions of a quaternionic variable}, Springer Monographs in Mathematics, Springer, Heidelberg, 2013.
\bibitem[GS]{GS} { G. Gentili, D. Struppa}, {\em A new theory of regular functions of a quaternionic variable}, Adv. Math., { 216}, 279--301, 2007.
\bibitem[GMP]{GMP} R. Ghiloni, V. Moretti, A. Perotti, {\em Continuous slice functional calculus in quaternionic Hilbert spaces}, Rev. Math. Phys. 25 (4), :1350006,83, 2013.
\bibitem[GP]{GP} {\sc R. Ghiloni, A. Perotti}, {\em Slice regular functions of several Clifford variables},
Proceedings of ICNPAA 2012 - Workshop ``Clifford algebras, Clifford analysis and their applications'', AIP Conf. Proc. 1493, 734-738, 2012
\bibitem[GR]{GR} {\sc R.C.  Gunning, H. Rossi}, { Analytic Functions of Several Complex Variables}, AMS CHELSEA PUBLISHING
American Mathematical Society. Providence, Rhode Island (2009)
\bibitem[HL]{hl} {\sc G.\ Henkin, J.\ Leiterer}
The Oka-Grauert principle without induction over the basis dimension.
Math.\ Ann.\ { 311}, 71-93 1998.
\bibitem[PV1]{PV1}
{\sc J.  Prezelj, F. Vlacci}  {\em An interpolation theorem for slice-regular functions with application to very tame sets and slice Fatou–Bieberbach domains in $\H^2$}. Annali di Matematica  201 (5), 2137--2159 (2022)
  {\tt https://doi.org/10.1007/s10231-022-01195-w}
\bibitem[PV2]{PV2}{\sc J.  Prezelj, F. Vlacci},  {\em Quaternionic Cartan Coverings and Applications}. {\tt ArXiv http://arxiv.org/abs/2405.08692}
\end{thebibliography}
\end{document}